\documentclass[preprint,10pt]{elsarticle}

\usepackage{amssymb}
\usepackage{amsfonts, amsmath}
\usepackage{graphicx}
\usepackage{tikz}
\graphicspath{ {./images/}}
\usepackage{caption}
\usepackage{subcaption}
\usepackage{amssymb}
\usepackage{amsthm}
\usepackage{float}

\newtheorem{theorem}{Theorem}[section]

\newtheorem*{remark}{Remark}
\newtheorem{proposition}{Proposition}

\usepackage{varwidth}
\usepackage[linesnumbered,ruled,vlined]{algorithm2e}
\DeclareCaptionFormat{myformat}{
  \begin{varwidth}{\linewidth}%
    \centering
    #1#2#3%
  \end{varwidth}%
}

\usepackage{geometry} 
\journal{Journal of Approximation Theory}

\begin{document}

\begin{frontmatter}

\title{Deep Univariate Polynomial and Conformal Approximation}

\author[uchic]{Kingsley Yeon}

\affiliation[uchic]{organization={Department of Statistics and CCAM, University of Chicago}, 
            city={Chicago},
            postcode={60637}, 
            state={IL},
            country={USA}}

\begin{abstract}
A \emph{deep} approximation is an approximating function defined by composing more than one \emph{layer} of simple functions. We study deep approximations of functions of one variable using layers consisting of low-degree polynomials or simple conformal transformations. We show that deep approximations to $|x|$ on $[-1,1]$ achieve exponential convergence with respect to the degrees of freedom. Computational experiments suggest that a composite of two and three polynomial layers can give more accurate approximations than a single polynomial with the same number of coefficients. We also study the related problem of reducing the Runge phenomenon by composing polynomials with conformal transformations.

\end{abstract}

\begin{keyword}
Composite polynomial \sep Kolmogorov-Arnold representation theorem\sep pth root \sep conformal map \sep function approximation \sep neural network approximation

\end{keyword}

\end{frontmatter}



\section{Introduction} \label{sec:intro}

It is commonly believed that the power of deep 
neural nets comes from composing ``layers'' of mathematical ``neurons''\cite{Devore}.
Each individual neuron is a relatively simple function.
Functions of many variables that solve practical optimality conditions seem
to be well approximated by compositions of such functions.
Is this true in the simpler situation of functions of one variable?
Is composing relatively simple univariate functions a powerful way to 
approximate important target functions?

We give some numerical experiments to compare approximations of function with {\em deep} polynomials and traditional linear least squares fixing the same degrees of freedom as well as analysis with 
elementary conformal transformations to show that composition can be 
a powerful way to build accurate approximations.

A {\em deep} polynomial is a polynomial that is created by composing two or more polynomials, 
which we refer to as layers. Deep univariate polynomial approximation has difficulties incommon with deep neural nets. One such difficulty 
is overcoming multiple local minima during {\em training} for effective optimization. The trial space 
of composite polynomials of fixed degrees is non-linear, so even the problem of least squares 
approximation can have multiple local minima similar to those of the complex landscapes in deep neural networks \cite{complex_land}. The Hessian of the objective function at a local or 
global minimum can be ill-conditioned, which suggests that the response surface is nearly flat in 
certain directions, and polynomials that are not close will give nearly identical approximation accuracy.

We present numerical experiments that show, for example, that the best approximation 
of a certain Bessel function by a composite of two degree four polynomials is much better
than the best approximation by a single polynomial with the same number
of ``tunable'' coefficients (degree seven, as explained below in Figure \ref{fig:bessel10}).

The reason for the success of deep approximations is a mystery. It is well known that composite polynomial form a small subclass of polynomial with the same degree \cite{Ritt, Rickards}.
To be sure, some functions $f$ are better approximated by a single 
polynomial of degree nine than by a composite of two polynomials of degree five.
For example, $f$ itself may be a polynomial of degree nine
that cannot be represented as a composite of two degree five polynomials. Equally important is the reverse case: a two-layer polynomial may be inexpressible by a single-layer polynomial with the same number of parameters. 
It may be that special functions like Bessel functions have good deep approximations because of their behavior in the complex plane. 

Gawlik and Nakatsukasa \cite{Gawlik} investigated a related concept involving the composition of rational functions to approximate the pth root ($x^{1/p}$) on the interval [0,1]. This approach demonstrated a doubly exponential convergence rate with respect to the degrees of freedom. Their work drew inspiration from and capitalized upon the notion of the recursive optimality of the Zolotarev function. This function yields the minimax rational approximant for square root and sign functions, as detailed in \cite{Gawlik_sqrt}. For the inverse pth root, deep polynomials exhibit exponential convergence with respect to the degrees of freedom, this includes the important subcase of $|x|$,  Section \ref{sec:pthroot}.

Section \ref{sec:conformalmap} explores composite approximations $f(x) \approx q(p(x))$
where $p$ is chosen as analytic mapping rather than a polynomial.
Clearly, the property of being well approximated in that sense is invariant under
conformal changes of variables. The section also studies the possibility of avoiding the 
Runge phenomenon in polynomial approximation using conformal maps as 
preconditioners.
The Runge phenomenon is the failure of high interpolation using high degree
polynomials on uniformly spaced points, even when the function $f$ is being interpolated
is analytic on the interval of interpolation.
This can be traced to singularities in $f$ in the complex plane close to the 
real interpolation interval.
The conformal preconditioner made those singularities further from the real interval
so that the polynomial in interpolates of the preconditioned functions converge.


\section{Deep polynomials, definition, training} \label{sec:dp}

A {\em deep} polynomial is a high-degree polynomial that is a composite of two or more lower-degree polynomials.
A degree $d$ polynomial of one variable takes the form
\[
      p(x) = a_d x^d + \cdots + a_0 \; .
\]
It is defined by $d+1$ coefficients $a_0$, $\ldots$, $a_d$.
We write $d = \text{deg}(p)$.
In principle, we should require that $a_d \neq 0$
but none of our numerical experiments do this explicitly.
Let $g$ be the composite polynomial
\[
    g(x) =  \left(q\circ p\right)(x) = q(p(x)) \; .
\]
The composite has degree $\text{deg}(g) = \text{deg}(p)\,\text{deg}(q)$.

A deep polynomial with $L$ layers is an $L$-fold composite:
\begin{equation} \label{pc}
       g(x) = p_1\circ p_2 \circ \cdots \circ p_L(x) \; .
\end{equation}
The {\em layers} are polynomials $p_k$ with degrees $d_k = \text{deg}(p_k)$.
The deep polynomial degree is the product of the degrees of the layers: 
\begin{equation}   \label{deg}
    \text{deg}(g) = D = d_1\cdot \;\cdots\;\cdot d_L \; .
\end{equation}
The composite degree $D$ can be large while the layer degrees, $d_k$, are moderate.
Subsection \ref{sec:normalization} shows that the number of free parameters defining 
$g$ in \eqref{pc} is $N=d_1+\cdots+d_L-L+2$.
This is usually much smaller than $D+1$, which is the number of free parameters
defining a general polynomial of degree $D$.
The general polynomial of degree $D$ cannot be expressed as a composite
with degrees $d_k$ if $N < D+1$.
Thus, the set of deep polynomials with degrees $d_k$ is a ``thin'' 
and nonlinear subset of the set of all polynomials of degree $D$.

We consider the problem of approximating a target function $f$ by a deep composite $g$
in the least squares sense on the interval $[-1,1]$.
The loss function for this is
\[
      F(p_1,\cdots,p_L) = \frac{1}{2}\int_{-1}^1 \left[ f(x) - g(x) \right]^2dx \; .
\]
Minimizing $F$ is not a linear least squares problem because the coefficients
of $p_k$ occur in the composite $g$ in a nonlinear way.

We describe this more concretely for $L=2$ layers.
For those, we use simpler notation.
The two layers of $g(x) = q(p(x))$ have degrees $d=\text{deg}(q)$
and $e=\text{deg}(p)$.
The coefficients are called $b_j$ and $a_j$:
\begin{align} 
    q(y) &= b_d y^d + \cdots + b_0       \label{q} \\
    p(x) &= a_e x^e + \cdots + a_0       \label{p} \; .
\end{align}
The loss function is
\begin{equation}  \label{F}
      F(a_0,\cdots,a_e,b_0,\ldots,b_d) = \frac{1}{2}\int_{-1}^1 \left[ \,f(x) - q(p(x))\,\right]^2 dx \;.
\end{equation}
The best $L^2$ approximation error is
\begin{equation}  \label{R}
    R_{de}(f) = 
        \min_{a_0,\ldots,b_d} \sqrt{2 F(a_0,\cdots,a_e,b_0,\ldots,b_d)}
\end{equation}
Figure \ref{fig:bescomerr} has a plot of $R_{de}$ for a specific target function $f$
as a function of $d$, with $d+e=D=10$ fixed.
The ``shallow'' approximations with $e=0$ or $d=0$ are seen to be much worse than
deep approximations with $d=e=5$.

\subsubsection{A Note on Chebyshev Basis}   
Switching from the monomial basis to the Chebyshev basis does not help in this context. For an approximation \( q(p(x)) \) on \([-1,1]\), this change would apply to \( p \) but not to \( q \), because the domain \( q \) needs to live on is unknown. Using the Chebyshev basis does not appear to simplify the optimization or improve performance in this setting.


\subsection{The gradient} \label{sec:grad}

We evaluate the loss function integral \eqref{F} and $L_2$ approximation error \eqref{R} using Gauss Legendre quadrature with $m=100$ points and weights.
The same points and weights are used for all integrals and are just beyond the minimum
needed for all such integrals in our experiments to be ``exact in exact arithmetic''.
The {\tt scipy} (version {\em 1.6.2}) routine {\tt scipy.special.rootslegendre}
was used to compute the points and weights.

The gradient of the loss function was calculated ``exactly'' using Gauss Legendre
quadrature applied to formulas that come from \eqref{F} by direct differentiation
using the definitions \eqref{p} and \eqref{q}. (i.e. exact in machine precision if $f$ is a polynomial and very well for analytic $f$):
\begin{equation}  \label{gradb}
	\frac{\partial F}{\partial b_j} 
	     = - \int_{-1}^1 \left[\,f(x) - q(p(x))\,\right]
	                    \cdot\left[\, p(x)\,\right]^j\,dx \; .
\end{equation}
\begin{equation}  \label{grada}
	\frac{\partial F}{\partial a_k}  
	   = - \int_{-1}^1 \left[\, f(x)-q(p(x))\,\right] 
		\cdot q'(p(x))\cdot x^k \,dx \; .
\end{equation}
The integrands of \eqref{gradb} and \eqref{grada} namely $q(p(x))\cdot p(x)^j$, $q(p(x))\cdot q'(p(x))\cdot x^k$ are polynomials whose degrees are 
at most $2de$.


\subsection{Normalization}      \label{sec:normalization}
The deep polynomial $g$ in \eqref{pc} does not uniquely determine the layer polynomials 
$p_m$.
For example, a composite of two degree 1 polynomials has degree 1.
If $p(x) = a_1x+a_0$ and $q(y) = b_1y+b_0$, then 
\[
     q(p(x)) = c_1 x + c_0 \; , \; \text{ with } c_1 = a_1b_1 \text{ and } c_0 = a_1b_0+a_0 \;.
\]
The four parameters defining $p$ and $q$ determine just two parameters, $c_1$ and $c_0$,
that define $g$.
Thus the optimal layer polynomials $p$ and $q$ are not unique even if the 
optimal deep polynomial $q(p(x))$ is unique.

We address this issue by showing that any composite $g=q\circ p$ may be achieved
as a composite $p = \widetilde{q}\circ \widetilde{p}$ so that the leading and trailing
coefficients of $\widetilde{p}$ are $\widetilde{a}_e = 1$ and $\widetilde{a}_0 = 0$.
The normalization fixes the two parameters $a_e$ and $a_0$.
This reduces the number of free parameters from $d+1+e+1=d+e+2$ to just $d+e$.
In the $d=e=1$ example, we can achieve any $g(x) = c_1x+c_0$ using $p(x) = x$
(i.e., $a_1=1$ and $a_0=0$) and $q(y) = c_1y+c_0$.

We found that the optimization worked much better with the normalization than without.
We replaced the un-normalized problem \eqref{R} with the equivalent normalized problem. For $L=2$ this is:
\begin{equation} \label{On}
      \min_{a_1,\cdots,a_{e-1},b_0,\cdots b_d} F(0,a_1,\cdots,a_{e-1},1, \; b_0,\cdots,b_d) \; .
\end{equation}
The expression \eqref{F} for $F$ and the gradient formulas \eqref{gradb} and
\eqref{grada} are unchanged.
The deflation strategy we use to avoid local minima relies on normalization.
Deflation need not and does not in our experiments, escape from degenerate local minima without normalization.

The result for the $L=2$ case above is Lemma 1 below.
The corresponding result for deeper $L>3$ composites is in Theorem 2.
The proof of Theorem 2 is an induction using Lemma 1.\\

\paragraph{Lemma 1}
Let $q(x), p(x)\in\mathbb{F}[x]$. Let $g(x) = q(p(x))$. Then there exists $\widetilde{q}(x), \widetilde{p}(x)$, such that $g = \widetilde{q}(\widetilde{p}(x))$, and $\widetilde{p}(x) = x^e + .... + \widetilde{a}_1x$, where $e = \text{deg}(p)$.

\paragraph{Proof}
Let $p(x) = a_ex^e + ... + a_0$, where $a_e \neq 0$. Then, define $p_1(x) = \frac{p(x)}{a_e}$. It is clear that we have:
\begin{equation*}
    \begin{split}
        g(x) &= q(p(x))\\
        &= \sum^d_{i=0}b_i(p(x))^i\\
        &= \sum^d_{i=0}b_i(a_ep_1(x))^i\\
        &= \sum^d_{i=0}b_ia^i_e(p_1(x))^i\\
    \end{split}
\end{equation*}
Thus, by defining $$q_1(y) = \sum^d_{i=0}b_ia^i_ey^i$$ we have $q_1(p_1(x)) = g(x)$. Now let $\widetilde{p}(x) = p_1(x) - \frac{a_0}{a_e}$. It is clear that $\widetilde{p}(0) = 0$. Then $p_1(x) = \widetilde{p}(x) + \frac{a_0}{a_e}$. Thus:
\begin{equation*}
    \begin{split}
        g(x) &= q_1(p_1(x))\\
        &= \sum^d_{i=0}b_ia^i_e \left(\widetilde{p}(x) + \frac{a_0}{a_e} \right) ^i\\
        &= \sum^d_{i=0}b_ia^i_e \left(\sum^i_{j=0}\binom{i}{j}\widetilde{p}(x)^j \left( \frac{a_0}{a_e} \right)^{i-j} \right)\\
        &= \sum^d_{j=0} \left(\sum^d_{i = j}b_ia^i_e\binom{i}{j} \left(\frac{a_0}{a_e}\right)^{i-j}\right)\widetilde{p}(x)^j\\
    \end{split}
\end{equation*}
By defining $$\widetilde{q}(y) = \sum^d_{j=0} \left(\sum^d_{i = j}b_ia^i_e\binom{i}{j} \left(\frac{a_0}{a_e}\right)^{i-j}\right)y^j$$ we have $\widetilde{q}(\widetilde{p}(x)) = g(x)$.
\qedsymbol

\paragraph{Remark}
It is not hard to see that $\text{deg}(\widetilde{q}) = \text{deg}(q)$, since $\text{deg}(g) = \text{deg}(q)\text{deg}(p) = \text{deg}(\widetilde{q})\text{deg}(\widetilde{p})$, and that $\text{deg}(p) = \text{deg}(\widetilde{p}) = e$.

\paragraph{Theorem 2} 
Let $p_i(x)\in \mathbb{F}[x]$, for $2\leq i\leq N$. Let $g(x) = p_1\circ p_2\circ...\circ p_N(x)$. Then there exists $\widetilde{p}_1,...,\widetilde{p}_N$, such that $g(x) = \widetilde{p}_1\circ \widetilde{p}_2\circ...\circ \widetilde{p}_N(x)$, $deg(p_i) = deg(\widetilde{p}_i)$, for all $1\leq i\leq N$, and $\widetilde{p}_j(x)$ has leading coefficient $1$ and constant term $0$, for all $2\leq j\leq N$.

\paragraph{Proof}
We prove by induction on $N$. The case $N = 2$ is proven in Lemma 1. Assuming the conclusion holds for all $N-1$ layers of composition, we try to prove it holds for $N$ layers. Let $p(x) = q_2\circ q_3\circ...\circ q_{N}(x)$. By inductive assumption, there exists $p_2,...,p_N$, such that $deg(p_i) = deg(q_i)$,  for all $2\leq i\leq N$, and $p_j(x)$ has leading coefficient $1$ and constant term $0$, for all $3\leq j\leq N$. Now let $y = p_3\circ...\circ p_{N}(x)$. Then we have $q_1(p(x)) =q_1(p_2\circ p_3\circ...\circ p_{N}(x)) =q_1(p_2(p_3\circ...\circ p_{N}(x)))=q_1(p_2(y))$. By the result of Lemma 1, we know there exists $\widetilde{q}_1, \widetilde{q}_2$, such that $deg(\widetilde{q}_1) = deg(q_1)$, $deg(\widetilde{q}_2) = deg(p_2)$, $q_1(p_2(y)) = \widetilde{q}_1(\widetilde{q}_2(y))$, and that $\widetilde{q}_2$ has leading coefficient $1$ and constant term $0$. Thus, take $\widetilde{q}_j = p_j$, for all $3\leq j\leq N$, we have $g(x) = \widetilde{q}_1\circ \widetilde{q}_2\circ...\circ \widetilde{q}_N(x)$, $deg(q_i) = deg(\widetilde{q}_i)$, for all $1\leq i\leq N$, and $\widetilde{q}_j(x)$ has leading coefficient $1$ and constant term $0$, for all $2\leq j\leq N$.
\qedsymbol

\begin{proposition}
If $p(x) = x^e + ... + a_1x$, i.e monic and zero constant-term, and $q(x) = b_dx^d + ... + b_0$, then $q(p(x))=g(x)$ and $q,p$ uniquely determined by $g$.  
\end{proposition}

We can see that the uniqueness proof follows from the constructive proof of Theorem 2.


\subsection{Optimization}      \label{sec:opt}

We use a somewhat ad-hoc hybrid optimization strategy that suffices for 
the computational experiments reported in Section \ref{sec:results}.
More sophisticated methods would be necessary to solve harder problems.
For example, one could use better-conditioned bases for polynomial spaces,
problem-specific preconditioning, an analytic Hessian matrix, etc.

\subsubsection*{Local optimization}

We compute $F$ from \eqref{F} and $\nabla F$ using Gauss Legendre quadrature 
as described in Subsection \ref{sec:grad}.
We use the optimization routine {\tt optimize.minimize(method=’BFGS’)} from {\tt scipy} version {\em 1.6.2}.
This implements a BFGS strategy that takes advantage of the analytic $\nabla F$.
Taking a convergence criterion intended to get the most accuracy possible
from the strategy, following the default successful termination criteria to be that gradient norm is less than gtol= ($10^{-12}$). We find that this tolerance is a good balance between the number of BFGS iterations and the accuracy of the solution. 

The next phase is a simple Newton strategy using a Hessian matrix estimated from 
the analytic $\nabla F$ which is computed with centered finite differences.
We use the BFGS output as the initial guess for simple Newton. Our affine invariant stopping criterion  is when $v^THv < 10^{-14}$, $v= H^{-1}\nabla F$, i.e $v$ is defined to be the Newton step and $H$ being the Hessian. This is observed to improve the results from the BFGS routine. 

The condition number of the finite difference Hessian is found to be quite large,
which explains the ability of the affine invariant Newton strategy to improve over
the BFGS routine, which is more affected by conditioning.

The optimization strategy described above is highly sensitive to the initial condition, we use random initial guesses drawn from a Gaussian distribution, $N(0,1)$. Following $n$ runs, the result is selected with the least $L_2$ error \eqref{F}. 

The problem of visiting the same local minima (a common occurrence) is tackled using the {\em deflation} strategy. 

The optimization is done accurately enough to convince us that we are indeed finding distinct local minima, which are significantly different such that it impacts the quality of the final approximation.

\subsubsection*{Deflation}

The deflation algorithm (see algorithm \ref{algo:def}) uses knowledge of previously found minima and constructs an iteration that avoids it while finding a new local minimizer. An explicit demonstration of the deflation algorithm is provided in Subsection \ref{sec:deflation_example}. Keeping the same notation as \cite{farrell}, we briefly review the algorithm. Deflation can be seen as an adaptation of Newton's method, where the gradient of the problem is altered. A deflation operator

\begin{equation}  \label{deflation op}
      M_{\alpha,\beta}(u;r)\equiv \frac{I}{||u-r||^\alpha} + \beta I,
\end{equation}
is defined to act on \(\nabla F(x)\), where $u = (a_1,\cdots,a_{e-1},b_0,\cdots b_d)$ is the input coefficients, $r$ is the known solution we are deflating from, $I$ is the appropriate identity matrix and $||\cdot||$ is the Euclidean vector norm, $\alpha\in \mathbb{R} \geq 1$ is the exponentiated-norm and $\beta \geq 0$ is the shift. The shift term is to ensure that the norm of the deflation residual doesn't go to zero when $||u-r|| \rightarrow \infty$ as $u \rightarrow \infty$ which happens in practice. One can easily see that the deflation operator does not introduce new zeros. Let

\begin{equation}  \label{deflation gradient}
      G(u) = M_{\alpha,\beta}(u;r)\nabla F(a_1,\cdots,a_{e-1},b_0,\cdots b_d)
\end{equation}

\begin{equation}  \label{deflation hessian}
      K(u) = D[G(u)]
\end{equation}
which is the derivative operator $D$ of the vector-valued function $G(u)$. In practice, $K$ is computed by center finite difference. We avoid calling $G$ the gradient and $K$ the Hessian as they are altered by the deflated operator and need not be. The deflation iteration is defined as 

\begin{align}  
      x_{k+1} &= x_k - s*p_k  \label{deflation newton}  \\ 
      K(x_k)p_k &= G(x_k) \label{deflation linear sys}
\end{align}
$s$ is the learning rate (usually set to 1), and the deflation iteration in \eqref{deflation newton} will require the solution to the linear system \eqref{deflation linear sys}. Matrix $K$ unsurprisingly also has a high conditioning number. We use the same affine invariant stopping criteria as our simple Newton which was described above.

The subsequent deflation operator is updated by multiplying the denominator of the first component in the first deflation operator \eqref{deflation op} by $||u-r_i||$ before raising it to $\alpha$. 


 Similar to the deflation paper \cite{farrell} our deflation algorithm is highly sensitive to the choice of the shift hyperparameters $\alpha, \beta$. We find that by varying $\alpha, \beta$, deflation can identify more solutions. 
 
 Finally, a brief overview of the deflation algorithm can be found in Algorithm \ref{algo:def}.

\begin{algorithm}[H]
\KwIn{$f$: function, $n$: number of inner coefficient, $j$: number of outer coefficient, nDef, eps, step, $\beta$: deflation parameter, $\alpha$: deflation parameter, perturb: a small number, init: initial condition, random: boolean)}
\KwOut{defCoeffs}
Check \linebreak
If random is true then set init to be samples from $N(0, 1)$
\linebreak
Otherwise, use inputted init\\
Feed init into scipy BFGS and its output to simple Newton\\
Construct deflation matrix $M_{\alpha, \beta}$ using $r_1$, the output of simple Newton, from formula \eqref{deflation op}\\
Compute $G$ and $K$ using formulas \eqref{deflation gradient}, \eqref{deflation hessian} \\
Solve linear system (\ref{deflation linear sys}) and start deflation iteration of (\ref{deflation newton}) with 
\[
x_0 = r_i + (\text{perturb}\approx 0.001)
\] \\
Feed output of deflation iteration into BFGS, and its output to simple Newton. This will return our new minimizer $r_{i+1}$\\
Update $M$ with new found minimizer $r_{i+1}$ and thus recompute new $G$ and $K$\\
Repeat steps 5 to 7, for desired number of nDef
\caption{Deflation: defmulti in deflation.py file}
\label{algo:def}
\end{algorithm}


\section{Composite approximation to the inverse \(p\)th root}
\label{sec:pthroot}

The absolute value function \(|x|\) has been a central focus in approximation theory. A key milestone in this area is Newman’s result, which established that rational functions can achieve root-exponential convergence \(O(e^{-\sqrt{n}})\) when approximating \(|x|\) \cite{Newman1964}, whereas polynomial approximations are limited to a slower convergence rate of \(O(1/n)\). We have demonstrated that composite polynomials, by leveraging their internal structure, can achieve exponential convergence \(O(e^{-n})\) with respect to their degrees of freedom when approximating the inverse \(p\)th root. This framework is not only applicable to \(x^{-1/p}\) but also extends to other functions, including \(|x|\), which we will explore in detail.  

Beyond its theoretical significance in polynomial and rational approximation, \(|x|\) plays an important role in computational mathematics and numerical computing. Through the identity \(\text{sgn}(x) = \frac{x}{|x|}\) for \(x \neq 0\), the sign function emerges as a key component in fundamental non-arithmetic operations, such as comparison and the max function. The comparison operation, given by \(\text{comp}(u, v) = \frac{1}{2}(\text{sgn}(u-v) + 1)\), is widely used in machine learning algorithms like support vector machines, where it plays a role in classification and data separation. Similarly, the max function, defined as \(\text{max}(u, v) = \frac{(u+v) + (u-v) \, \text{sgn}(u-v)}{2}\), is crucial in convolutional neural networks for feature extraction through max pooling. Turns out, the optimal (non-scalar multiplication) way to approximate \(\text{sgn}(x)\) is by constructing a composite polynomial, as done by Lee et al. \cite{Lee}.

The problem of representing functions like \(|x|\) naturally extends to a broader class involving non-integer powers. For \(q/p \in \mathbb{Q}\), the expression \((x^{1/p})^q\) provides a structured way to approximate more general function classes. Zhao and Serkh \cite{Zhao} expanded on this idea, showing that functions of the form  
\[
f(x) = \langle \sigma(\mu), x^\mu \rangle
\]
defined over \([0,1]\), where \(\sigma(\mu)\) is a distribution supported on \([a,b]\) with \(0 < a < b < \infty\), can be approximated with exponential efficiency using series of non-integer powers. These approximations have important applications, particularly in boundary integral equations (BIE), where domains with corners pose significant challenges \cite{Serkh}. In a similar spirit to how Newman's result on rational functions inspired the lightning method \cite{light}, we hope that composite polynomial approximations will serve as an effective tool for addressing functions with challenging singularities, including \(|x|\).

The direct approach of applying Newton's iteration for \(F(f) = f^p - x\) to compute \( f^* = x^{1/p}\) will yield a composite rational approximation to \(f^p\). Instead of Newman's root-exponential convergence rate, we now achieve exponential convergence with respect to the degrees of freedom. This result is by no means sharp; in fact, it has been shown that composite rational functions can attain doubly-exponential convergence \cite{Gawlik}.

I showed a manuscript of this paper to Nick Trefethen, who then contacted Jean-Michel Muller, an expert in elementary function approximation. Muller suggested that a deep polynomial approximation to \(\lvert x \rvert\) could be obtained by applying Newton iteration to the function \(\frac{x^2}{\sqrt{x^2}}\). Following this idea, we employ the well-known Newton iteration for the matrix inverse \(p\)th root \cite{Higham}:
\begin{equation} \label{e:newtoniter}
    F(f) \;=\; \frac{1}{f^p} \;-\; x, 
    \quad
    f_{k+1} \;=\; \frac{1}{p}\, f_k \Bigl[(p+1) - f_k^p\,x \Bigr],
    \quad
    x_0 = 1.
\end{equation}

\begin{theorem}[Exponential convergence of composite polynomial to \(\lvert x\rvert\)]
\label{thm:linConvCompPolynom|x|}
Consider the special case \(p=2\) and let \(x \in [-1,1]\), including \(x=0\). Define the iteration
\[
f_{k+1} 
\;=\; 
\tfrac12\, f_k \Bigl( 3 \;-\; x^2\, f_k^2\Bigr), 
\quad f_0 = 1.
\]
Then \(\{f_k\}\) converges linearly to \(\frac{1}{\lvert x\rvert}\) for \(x \neq 0\), and multiplying by \( x^2 \) gives us \(\tfrac{x^2}{\lvert x\rvert} = \lvert x\rvert\). For \(x=0\), the limit is trivially \(\lvert x\rvert = 0\). In particular, if we define 
\[
r_k \;=\; \lvert x\rvert\, f_k 
\quad\text{and}\quad 
E_k \;=\; 1 - r_k,
\]
then there exists a constant \(\alpha \le \tfrac58\) such that 
\[
E_{k+1} \;\le\; \alpha \, E_k,
\]
for all \(k\) beyond some finite index. Hence 
\(\bigl\{ f_k \bigr\}\) converges to \(\frac{1}{\lvert x\rvert}\) at a geometric (exponential) rate with ratio at most~\(\tfrac12\).
\end{theorem}

\begin{proof}
For \(x=0\), the iteration reads \(f_{k+1} = \tfrac32\,f_k\) with \(f_0=1\), and so \(f_k\) diverges to \(+\infty\), matching the fact that \(1/|x|\) is formally infinite when \(|x|=0\).  Therefore, the limit statement is satisfied in this trivial case.  

For \(x \neq 0\), define \(r_k = |x|\, f_k\).  Substituting \(f_k = \tfrac{r_k}{|x|}\) into the iteration 
\[
f_{k+1} 
\;=\;
\tfrac12\, f_k\, \bigl(3 - x^2 f_k^2\bigr)
\]
gives
\[
r_{k+1} 
\;=\;
\tfrac12\,r_k\,\bigl(3 - r_k^2\bigr).
\]
It is direct to check that \(r=1\) is a fixed point and that \(r_0 = |x|\in (0,1]\).  Whenever \(0<r_k<1\), we have \(0<r_{k+1}<1\) and \(r_{k+1}>r_k\), so \(\{r_k\}\) is strictly increasing and bounded above by~\(1\).  It follows that \(r_k\to 1\).  

To see the linear (geometric) convergence, set \(E_k = 1 - r_k\).  Then a direct computation using
\[
1-r_{k+1}
\;=\;
1-\tfrac12\,r_k\,\bigl(3-r_k^2\bigr)
\;=\;
1-\tfrac32\,r_k + \tfrac12\,r_k^3
\]
shows that 
\[
E_{k+1} 
\;=\;
1 - r_{k+1}
\;=\;
\tfrac12\,\bigl[2 - 3\,r_k + r_k^3\bigr]
\;=\;
\tfrac12\,\bigl[(1 - r_k)\,\bigl(2 - r_k - r_k^2\bigr)\bigr]
\;=\;
\alpha_k\,E_k,
\]
where 
\[
\alpha_k 
\;=\;
1-\tfrac12\,r_k\,\bigl(1 + r_k\bigr).
\]
Since \(r_k\in(0,1)\), it follows \(\alpha_k \in (0,1)\).  Moreover, once \(r_k\) grows beyond some fixed threshold in \((0,1)\), say \(r_k \ge \tfrac12\), we obtain
\[
\alpha_k 
\;=\;
1 - \tfrac12\,r_k\,(1+r_k)
\;\le\;
1 - \tfrac12\cdot \tfrac12\cdot \tfrac32
\;=\;
1 - \tfrac38
\;=\;
\tfrac58,
\]
and thus beyond that finite index we have \(E_{k+1}\le\tfrac58 E_k\).  By induction, \(\{E_k\}\) converges to \(0\) exponentially fast.  Since \(f_k = \tfrac{r_k}{|x|}\) and \(r_k\to 1\), it follows that \(f_k \to \tfrac1{|x|}\) and that \(\{f_k\}\) converges at a geometric rate.  This completes the proof.
\end{proof}

\begin{remark}
The determination of a uniform constant governing the geometric rate of convergence remains an open problem. Further analysis in this direction would be of significant interest.
\end{remark}

A similar analysis can be done using equation \eqref{e:newtoniter} to show exponential convergence of deep polynomial to the inverse \(p\)-th root. 

\begin{figure}[H]
\begin{centering}
\captionsetup{format=myformat}
	\includegraphics[scale= 0.5]{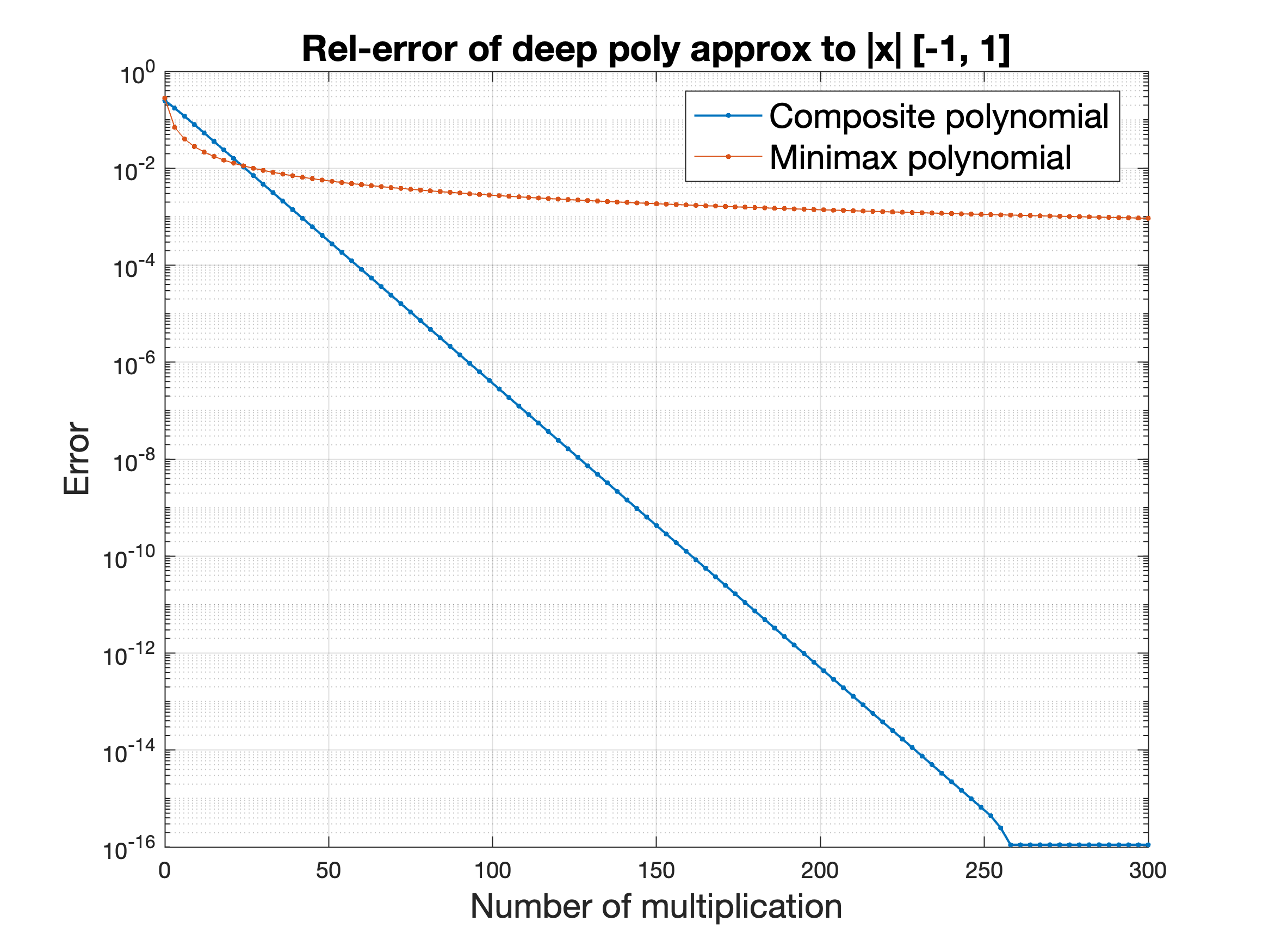}
	\caption{Relative error of the deep polynomial approximation to \(\lvert x\rvert\) on \([-1,1]\). The exponential decay of the error aligns with the theoretical linear convergence rate.}
 	\label{fig:deep_abs_err}
\end{centering}
\end{figure}

\begin{figure}[H]
\begin{centering}
\captionsetup{format=myformat}
	\includegraphics[scale= 0.5]{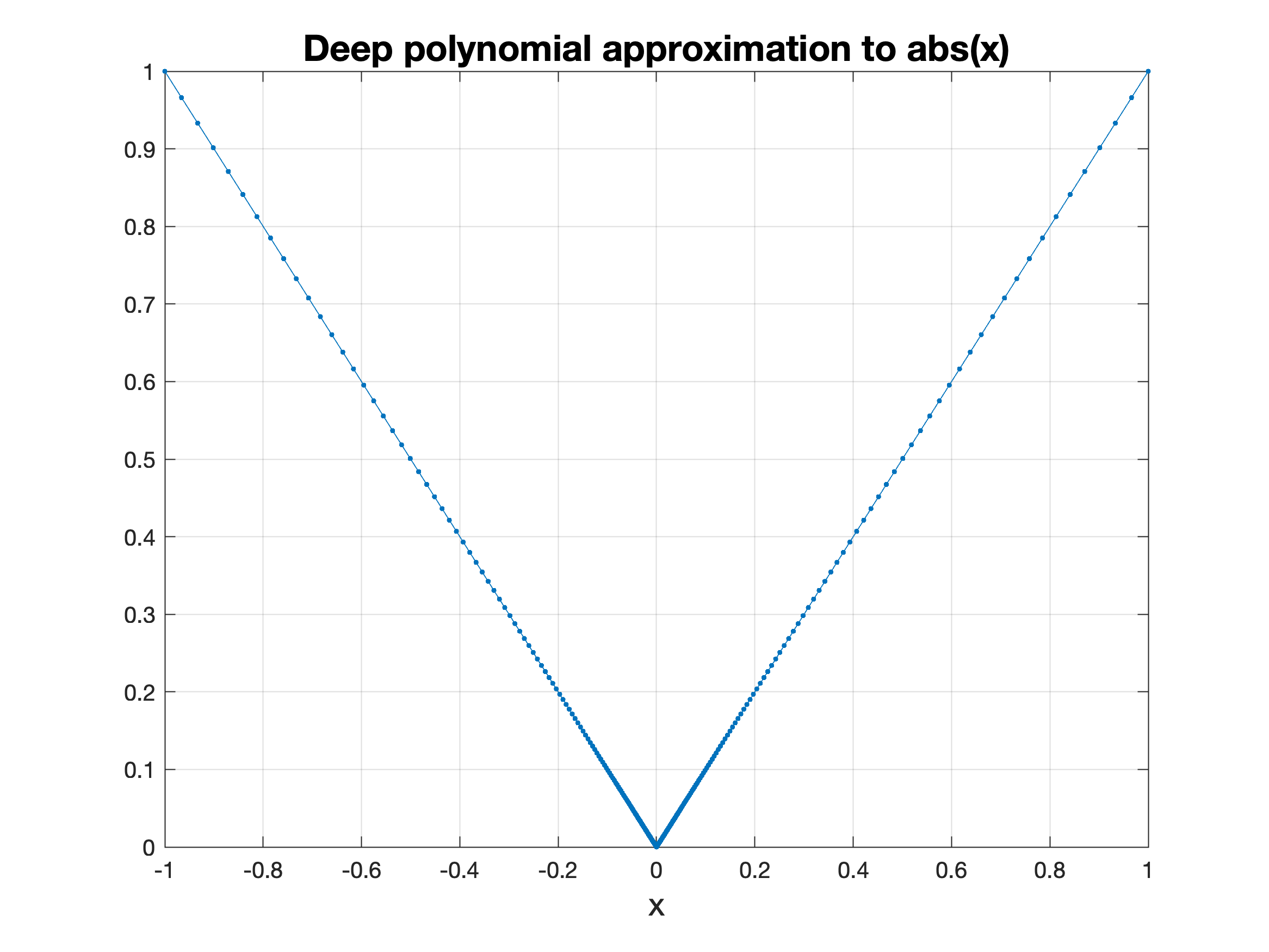}
	\caption{Deep polynomial approximation to \(\lvert x\rvert\) on \([-1,1]\). The graph illustrates accurate convergence to \(\lvert x\rvert\) across the interval, including endpoints and zero.}
 	\label{fig:deep_abs}
\end{centering}
\end{figure}

The degree of the composite polynomial \( f_k \) grows exponentially with each iteration of the Newton method, satisfying the recurrence relation \( d_{k+1} = 3d_k + 2 \) with \( d_0 = 0 \). Solving this explicitly gives \( d_k = 3^k - 1 \), demonstrating that the degree of \( f_k(x) \) increases exponentially with \( k \).

The current iteration \( f_k \) represents a composite polynomial approximation of the inverse \( p \)th root, \( \frac{1}{x^{1/p}} \). However, a deep polynomial approximation for the \( p \)th root itself, \( x^{1/p} \), remains unknown.


\section{Examples} \label{sec:results}

We compare the deep polynomial approximation with its single-layer counterpart by ensuring that the number of optimized coefficients stays the same. For instance, in Figure \ref{fig:2runge}, we compose two polynomials with 5 coefficients each, and after normalization, our single-layer counterpart has $5+5-2=8$ coefficients. The problem of minimizing the least squares error of a polynomial is convex, while the loss function associated with deep polynomial approximation resides on a non-linear manifold. We hope that this difference in manifold structure allows deep polynomials the ability to generalize over a larger class of functions; like a space-filling curve.

Consequently, it is challenging to locate the global minimum of a fixed degree deep polynomial. In a subsection \ref{sec:ensemble}, we will demonstrate numerically how optimization of finding such polynomials is difficult by employing an ensemble of random initial conditions to explore the optimization terrain.

We call an $L^\text{th}$ layer {\it deep} polynomial of type ($\mu_1, \mu_2, \dots, \mu_L$) when the $i^{\text{th}}$ polynomial layers have $\mu_i$ numbers of coefficients, recall that the inner layers $\mu_2, \dots, \mu_{L}$ are normalized, therefore the composite polynomial have $\mu_1 + \mu_2 + \cdots + \mu_L - 2(L-1)$ degrees of freedom. Thus the corresponding single-layer polynomial has the same numbers of coefficients as the degrees of freedom of the {\it deep} approximates.


\subsection{Example 1: Logistic $ \tanh(\alpha x) = \frac{e^{\alpha x}-e^{-\alpha x}}{e^{\alpha x}+e^{-\alpha x}}$ interval of $(-1,1)$} \label{sec:logistic}

Our first example examines the approximation of the sigmoid activation function commonly used in machine learning. For large \(\alpha > 0\), $\tanh(\alpha x) \approx \text{sgn}(x)$ and we have shown that deep polynomial approximations achieve exponential efficiency for $|x|$ and hence $\text{sgn}(x) = \frac{x}{|x|}$. To investigate whether this efficiency holds for smaller \(\alpha\), and to assess the effectiveness of our optimization algorithm in identifying such approximations, we set \(\alpha = 3\) and solve the minimization problem through 10 random trials.

Using a two-layer type \((4,4)\) approximation where both polynomials \(p\) and \(q\) are of degree 3 we achieved an accuracy of \(10^{-4}\) in the \(L_2\) norm. In comparison, a single-layer approximation with six coefficients attains a precision of \(10^{-3}\) under the same norm.

Bernstein's theorem for steep polynomials states that for a polynomial \( p \) of degree \( n \),
\[
\max_{|z|=1} |p'(z)| \leq n \max_{|z|=1} |p(z)|.
\]
For \( \tanh(\alpha x) \), the derivative at zero is \( \alpha \). Consequently, a polynomial approximation must have a degree of at least \( \mathcal{O}(\alpha) \) to achieve a good fit. Since the degree of deep approximations grows exponentially with the number of compositions, one could expect deep polynomials to outperform standard polynomials for steep functions. We illustrated this with a deep approximation of type (15, 15), which achieves an \( L_2 \) error of \( 1.65 \times 10^{-3} \), compared to \( 1.33 \times 10^{-1} \) for the corresponding minimax approximation (see Figure \ref{fig:sig1515}).

\begin{figure}[H]
\begin{centering}
\captionsetup{format=myformat}
	\includegraphics[scale= 0.35]{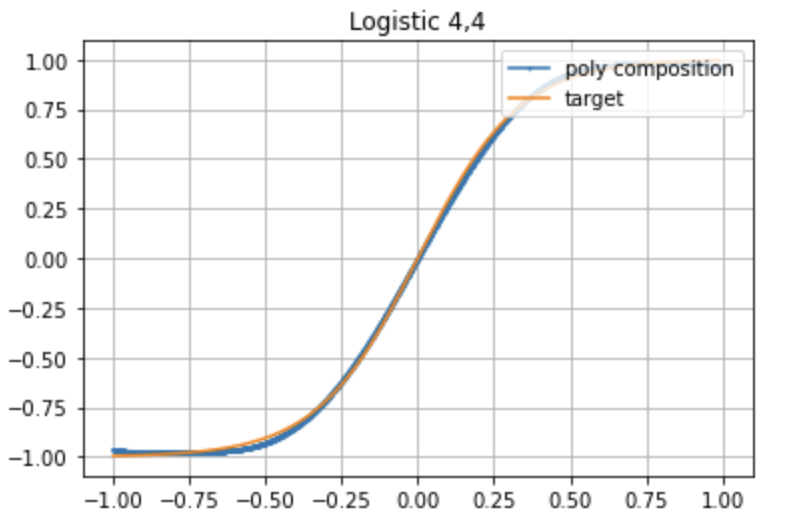}
	\caption{Logistic 10 random run, $L_2$ error: 2.411e-02, \\$p = 1x^3 -(2.5\cdot 10^{-5})x^2 - 1.94x + 0$\\
 $q = (4.43\cdot 10^{-1})x^3 -(2.92\cdot 10^{-7})x^2 -1.42x -(2.28\cdot 10^{-6})$ }
 	\label{fig:r1}
\end{centering}
\end{figure}

\begin{figure}[H]
\begin{centering}
\captionsetup{format=myformat}
	\includegraphics[scale= 0.40]{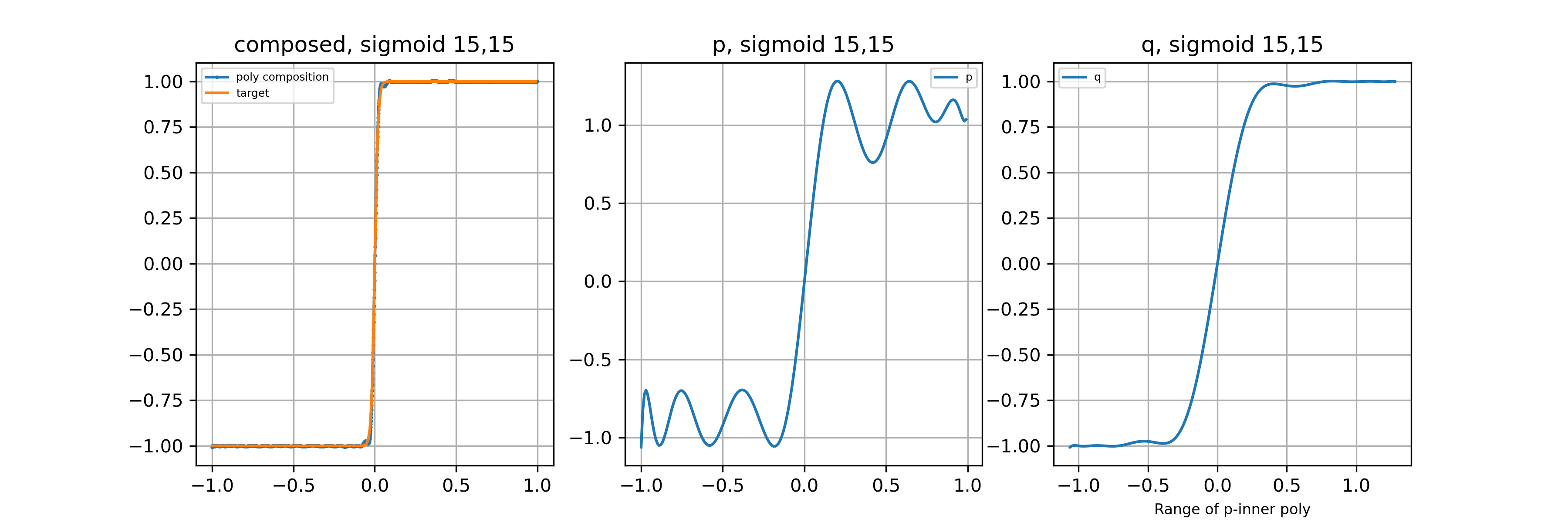}
	\caption{Sigmoid, $\alpha$ = 50, type (15, 15)}
 	\label{fig:sig1515}
\end{centering}
\end{figure}


\subsection{Example 2: Runge $\frac{1}{1+25x^2}$} \label{sec:runge}
Carl Runge presented an example of an analytic function for which polynomial interpolation diverges when equispaced nodes are used. This function exhibits distinctive properties in the complex plane. To evaluate the performance of deep polynomials, we approximate it using a type (5,5) approximation. Even for this relatively simple function, the optimization process is highly sensitive to the initial conditions. Figure \ref{fig:2runge} shows the best result among 10 trials. The \( L_2 \) error reveals the presence of multiple local minima, specifically at \( 2.73 \times 10^{-2} \), \( 5.7 \times 10^{-2} \), and \( 8.7 \times 10^{-2} \) (see Table \ref{table:rungeerror} and Figure \ref{fig:2runge}).

\begin{table}[!ht]
\centering
\begin{tabular}{||c | c||} 
 \hline
 Iteration & $L_2$ error \\ [0.5ex] 
 \hline\hline
 1 & 9.6448e-02 \\ 
 2 & 2.7334-02  \\
 3 & 2.7342e-02  \\
 4 & 2.9523e-02  \\
 5 & 8.8202e-02  \\ 
 6 & 5.7279e-02  \\
 7 & 5.6871e-02 \\
 8 & 8.7658e-02  \\
 9 & 2.7334e-02  \\
 10 & 8.7835e-02  \\ [1ex] 
 \hline
\end{tabular}
\caption{Runge $\text{deg}(p)=\text{deg}(q)=4$, $n=10$ random start $L_2$ error}
\label{table:rungeerror}
\end{table}

\begin{figure}[H]
\centering
\begin{subfigure}{0.4\textwidth}
  \centering
  \fbox{\includegraphics[width=0.9\linewidth]{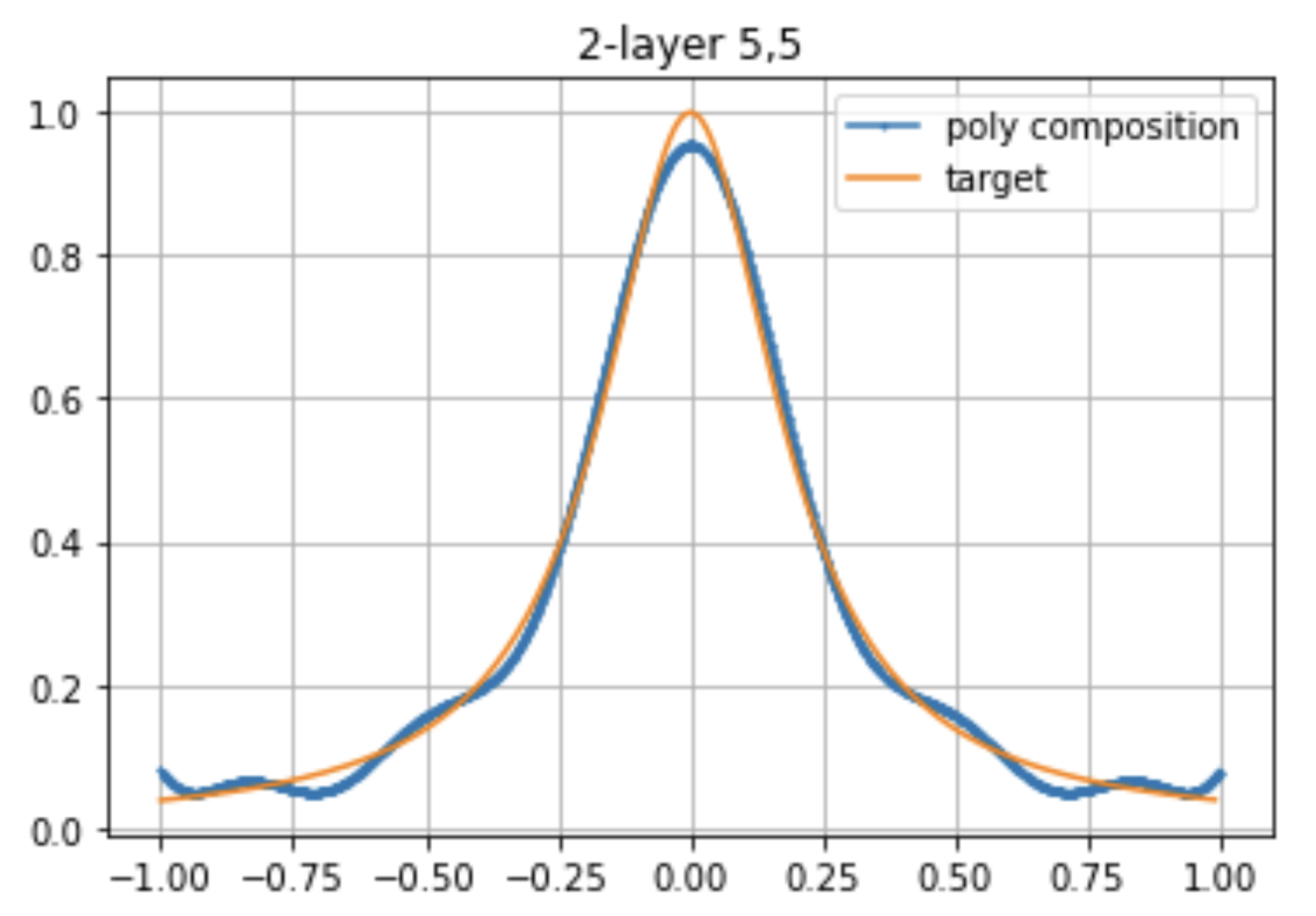}}
   \caption{Runge 5,5 $L_2$ error: 2.73e-02}
\end{subfigure}%
\begin{subfigure}{0.4025\textwidth}
  \centering
  \fbox{\includegraphics[width=0.89\linewidth]{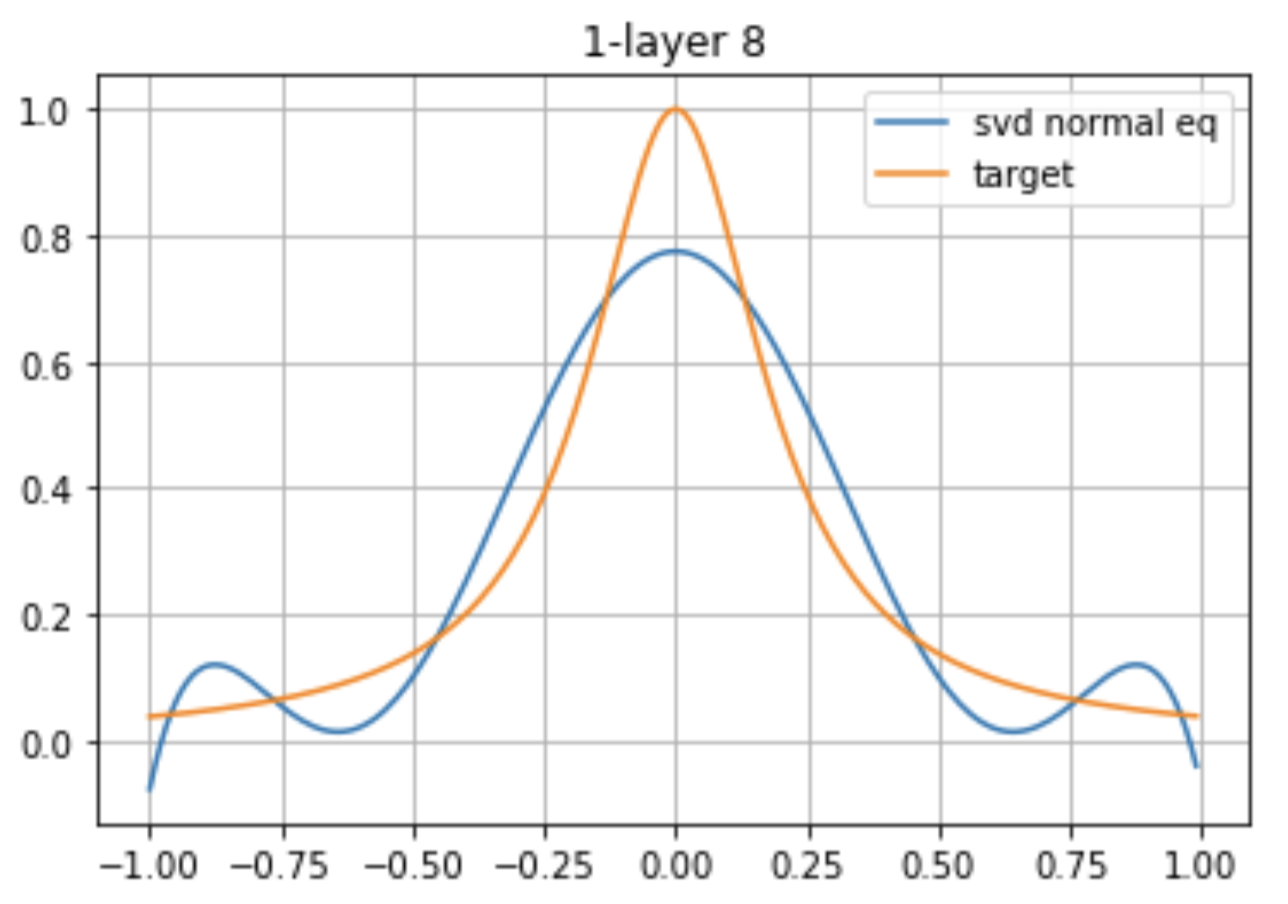}}
  \caption{Runge 8 $L_2$ error: 1.24e-01}
\end{subfigure}
\caption{Runge 2 layer vs 1 layer}
\label{fig:2runge}  
\end{figure}

In the 3-layer regime, we also encounter the issue of multiple local minima however our resulting error for composing 3, deg(4) polynomial with normalization is much better. (See Figure \ref{fig:3runge})

\begin{figure}[H]
\centering
\begin{subfigure}{0.4\textwidth}
  \centering
  \fbox{\includegraphics[width=0.9\linewidth]{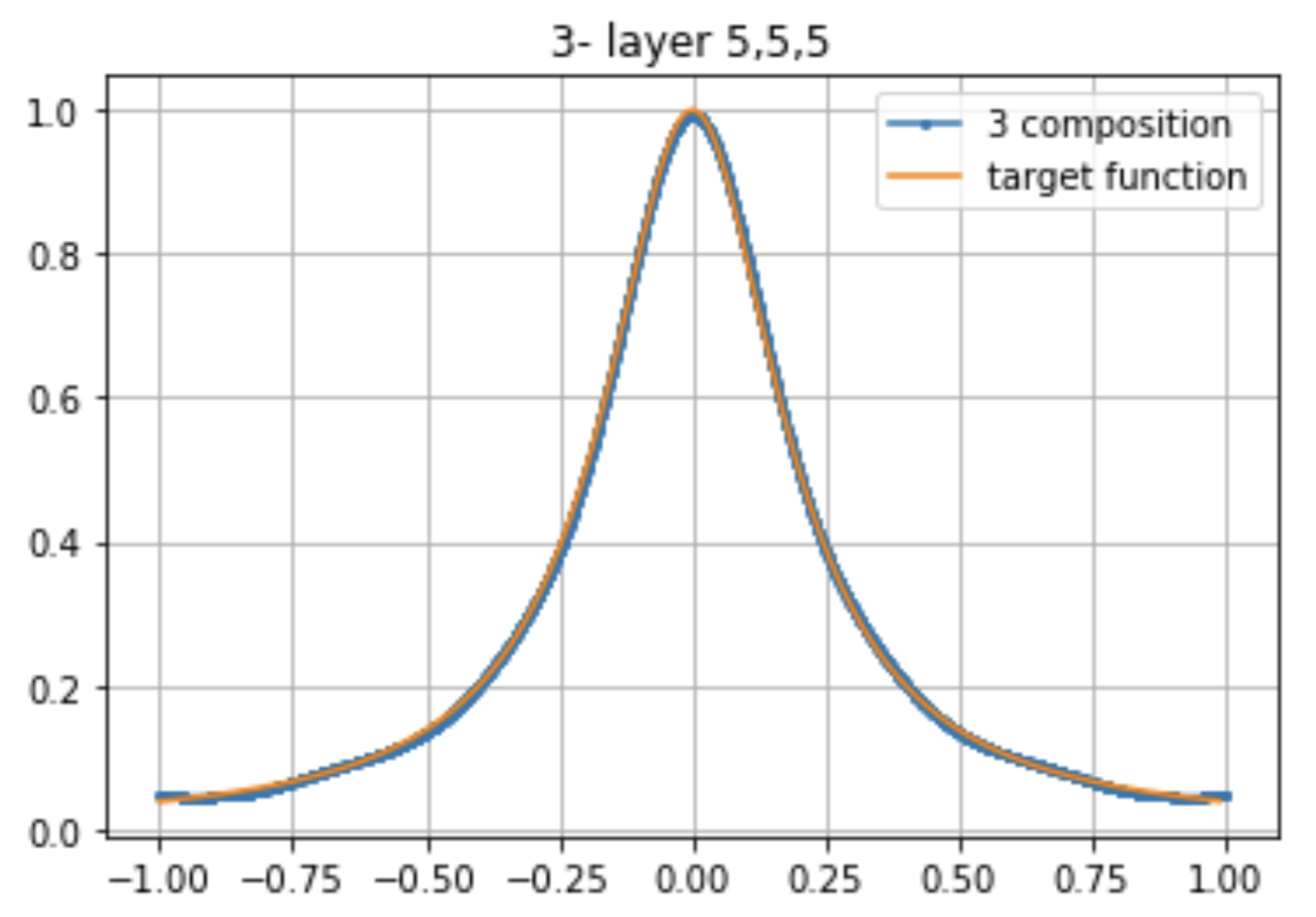}}
   \caption{Runge 5,5,5 $L_2$ error: 2.97e-03}
\end{subfigure}%
\begin{subfigure}{0.4\textwidth}
  \centering
  \fbox{\includegraphics[width=0.9\linewidth]{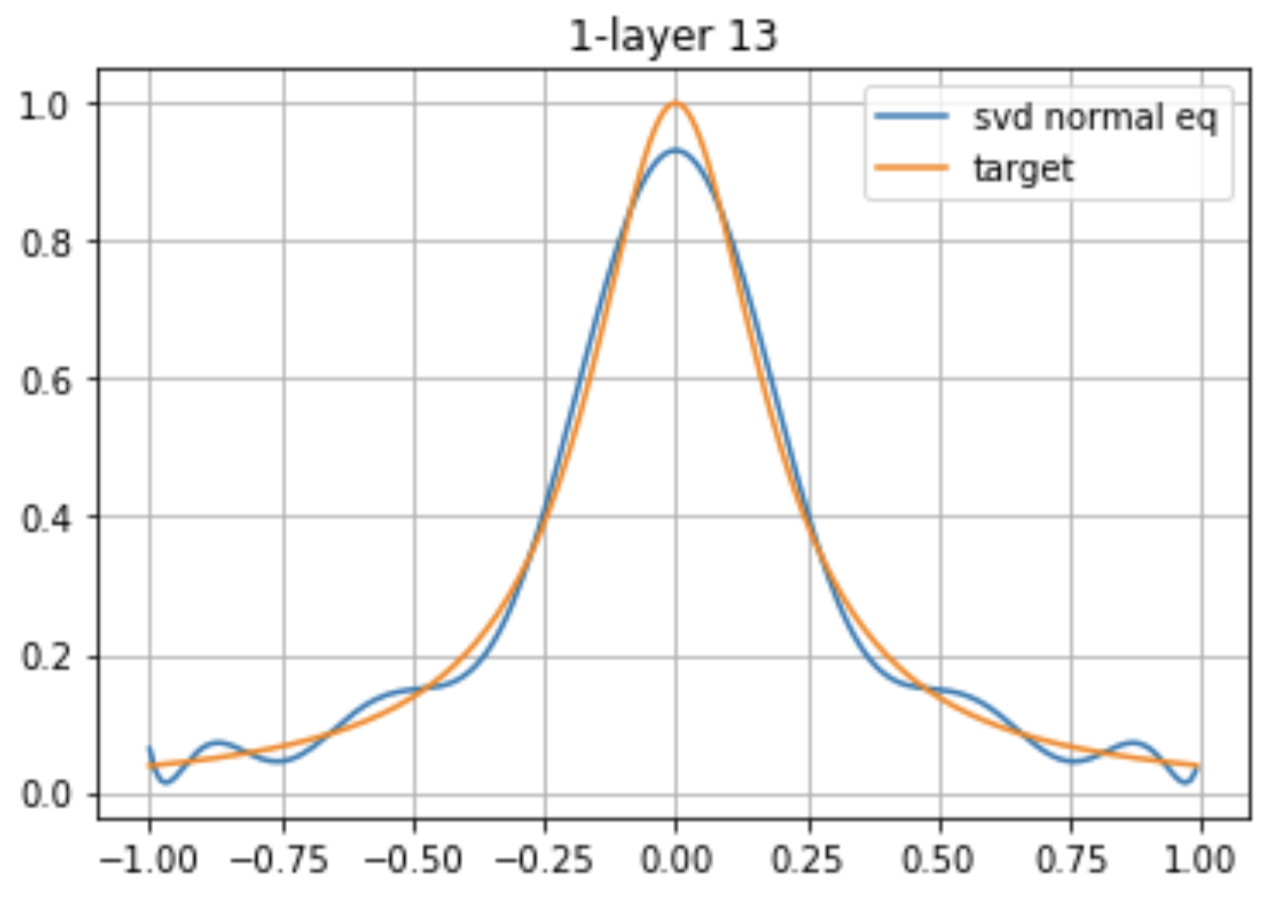}}
  \caption{Runge 13 $L_2$ error: 3.79e-02}
\end{subfigure}
\caption{Runge 3 layer vs 1 layer}
\label{fig:3runge} 
\end{figure}


\subsection{Example 3: Bessel function}

Bessel functions are important as they arrises as the separation of variable solution to Laplace and Helmholtz equation in cylindrical and spherical coordinates. The result of our Bessel approximation is surprisingly accurate in Figure \ref{fig:bessel10} with just 10 random initial condition trails. Admittedly, we could further speed up the computation by noticing the even symmetry around the origin of both Runge and Bessel, hence approximating only the interval (0,1).

\begin{figure}[H]
\centering
\begin{subfigure}{0.4\textwidth}
  \centering
  \fbox{\includegraphics[width=0.9\linewidth]{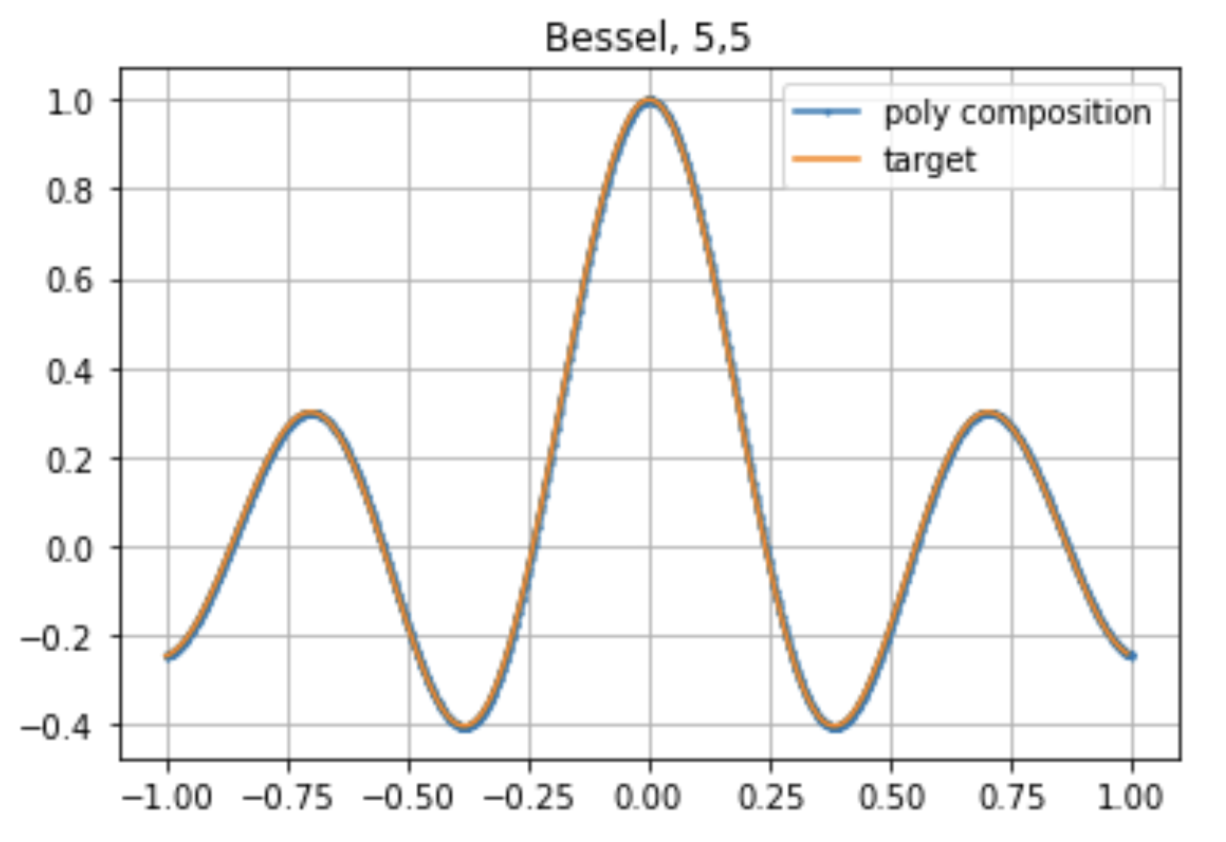}}
   \caption{Bessel 5,5 $L_2$ error: 1.02e-03}
\end{subfigure}%
\begin{subfigure}{0.409\textwidth}
  \centering
  \fbox{\includegraphics[width=0.906\linewidth]{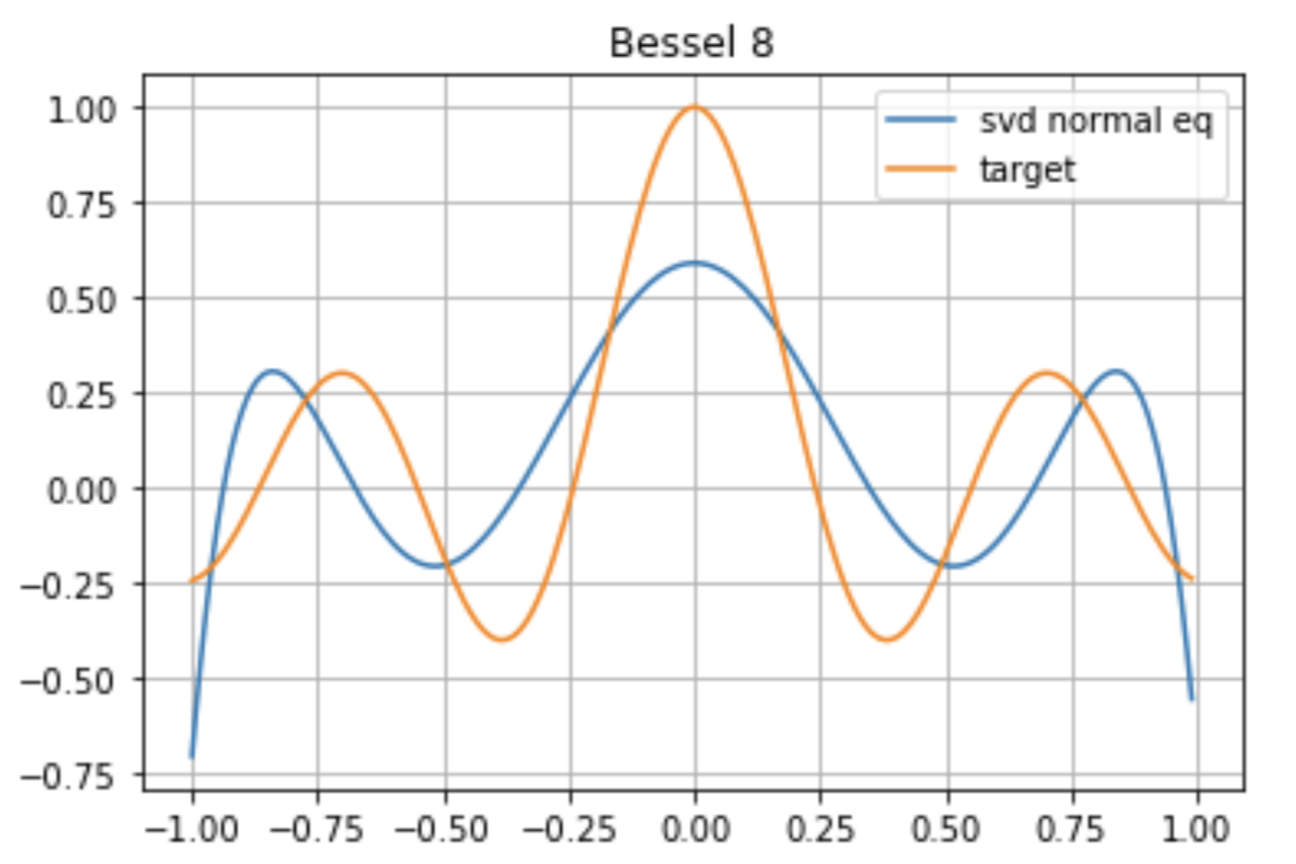}}
  \caption{Bessel 8 $L_2$ error: 3.57e-01}
\end{subfigure}
\caption{$J_{0}(10x)$, 2 layer vs 1 layer}
\label{fig:bessel10}  
\end{figure}

\begin{figure}[H]
\centering
\begin{subfigure}{0.402\textwidth}
  \centering
  \fbox{\includegraphics[width=0.9\linewidth]{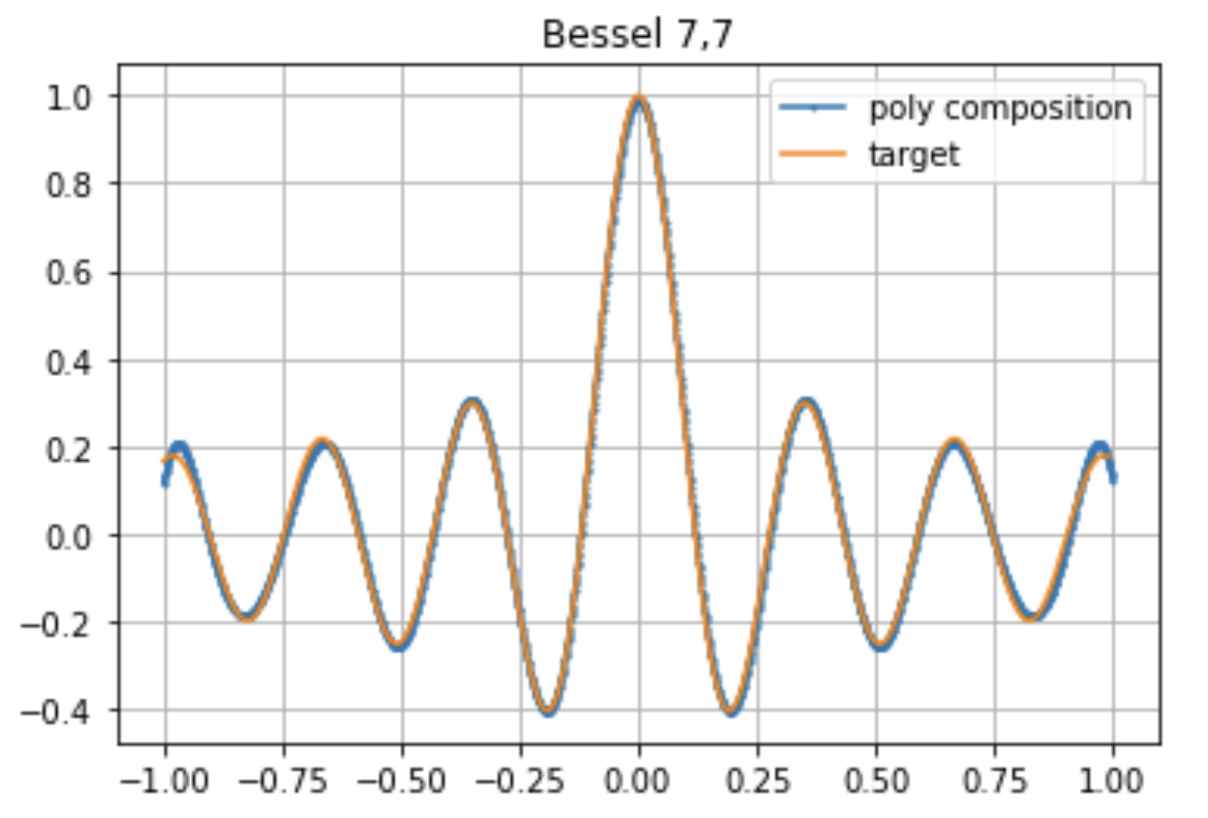}}
   \caption{Bessel 7,7 $L_2$ error: 1.53e-02}
\end{subfigure}%
\begin{subfigure}{0.4\textwidth}
  \centering
  \fbox{\includegraphics[width=0.9\linewidth]{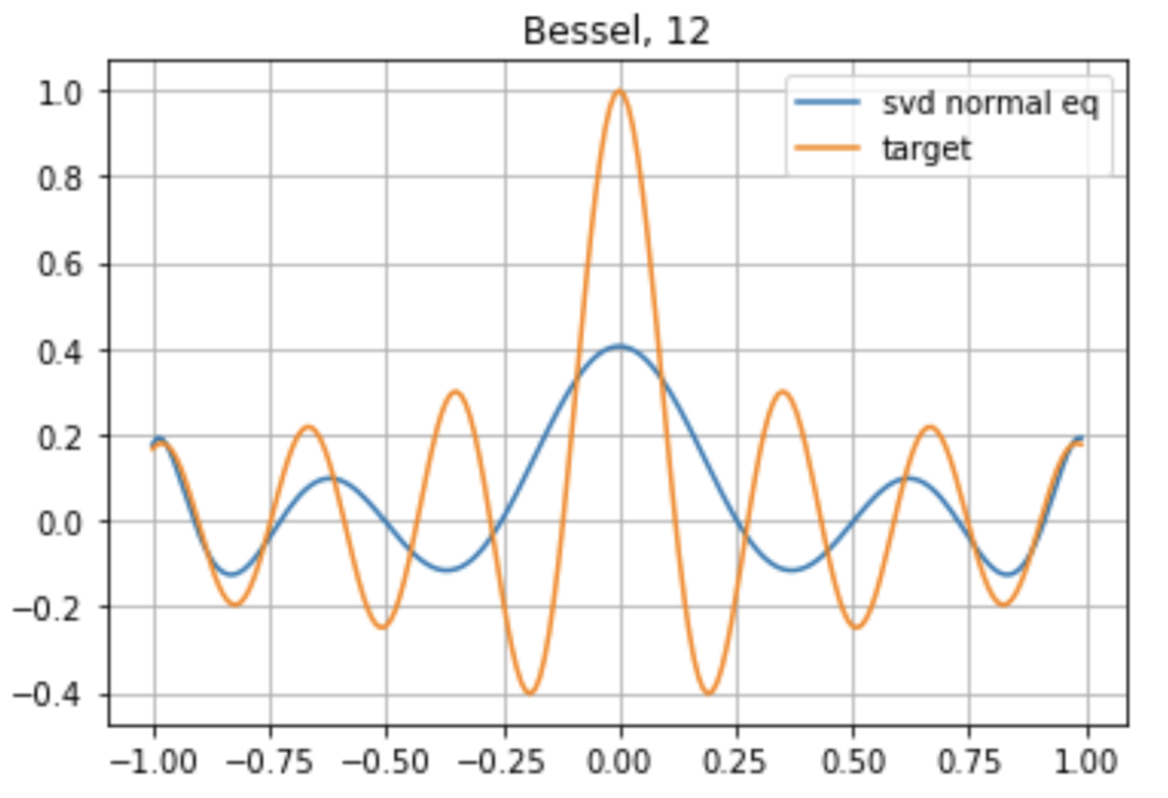}}
  \caption{Bessel 12 $L_2$ error: 3.60e-01 }
\end{subfigure}
\caption{$J_{0}(20x)$, 2 layer vs 1 layer}
\label{fig:bessel20}  
\end{figure}

\begin{figure}[H]
\centering
\begin{subfigure}{0.402\textwidth}
  \centering
  \fbox{\includegraphics[width=0.9\linewidth]{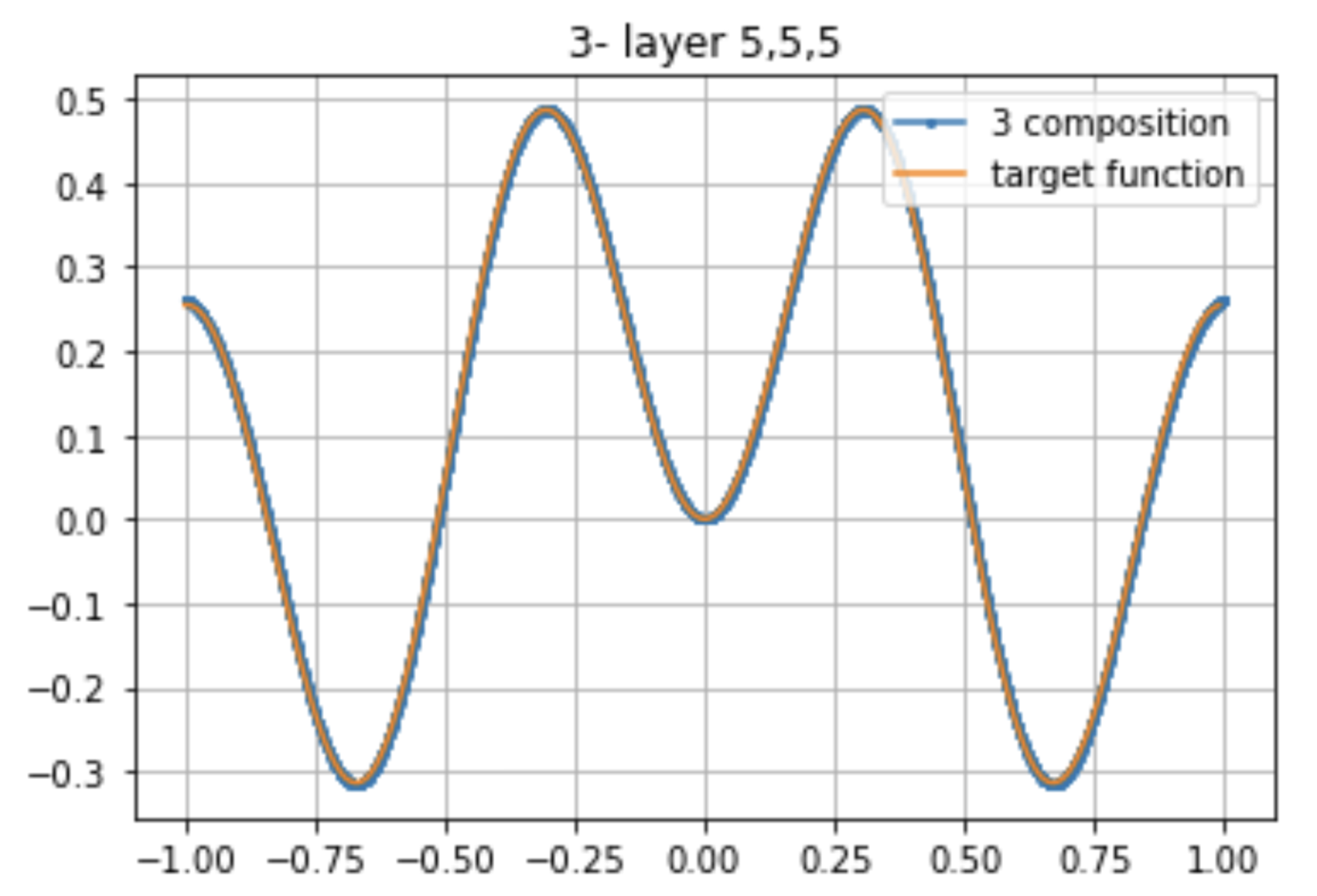}}
   \caption{Bessel 5,5,5 $L_2$ error: 7.15e-04}
\end{subfigure}%
\begin{subfigure}{0.4\textwidth}
  \centering
  \fbox{\includegraphics[width=0.905\linewidth]{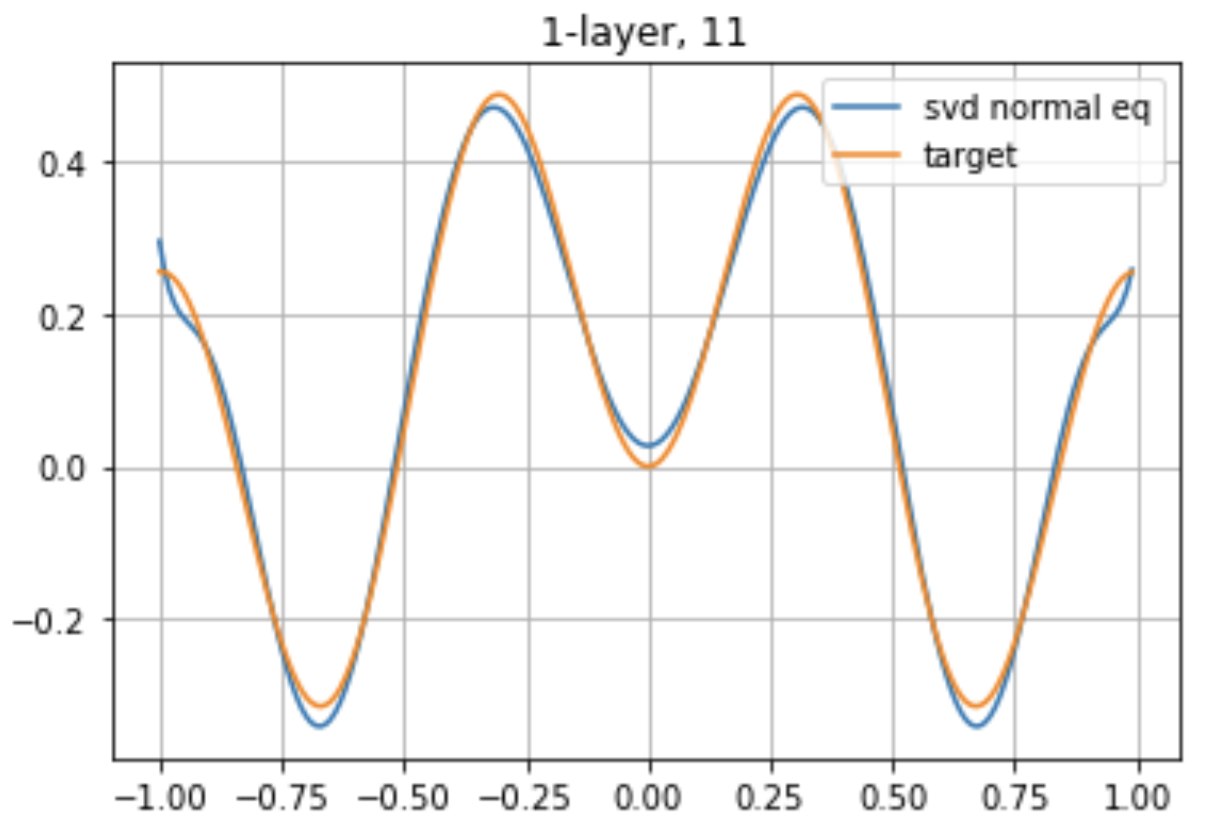}}
  \caption{Bessel 11 $L_2$ error: 2.78e-02  }
\end{subfigure}
\caption{$J_{2}(10x)$, 3 layer vs 1 layer}
\label{fig:bessel3layer}
\end{figure}


\subsubsection{Deflation with Bessel function}  \label{sec:deflation_example}

We illustrate the deflation algorithm using Bessel function with fixed initial condition\footnote{Initial condition used for deflation in section \ref{sec:deflation_example}: $-0.269553, 1.757204,$ \\ $0.509716, 1.428677, -1.660497, 1.703788, -2.291055, 0.557481 $.}. Figure \ref{fig:deflation1} shows the result after 1 deflation iteration. We see that the error has decreased and the 2 individual polynomials $p, q$ are different.

Deflation, combined with random initial conditions, proves to be a useful tool for problems with persistent suboptimal basins.

\begin{figure}[H]
\begin{centering}
\captionsetup{format=myformat}
	\fbox{\includegraphics[scale= 0.4]{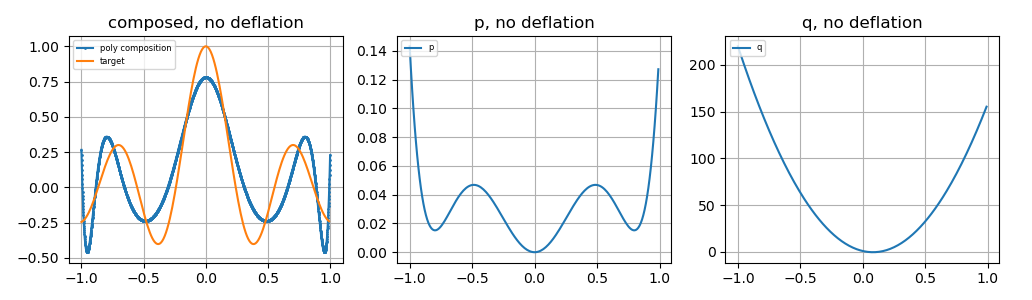}}
	\caption{Deflation 0th, $L_2$ error: 2.70e-01 }
 	\label{fig:deflation0}
\end{centering}
\end{figure}

\begin{figure}[H]
\begin{centering}
\captionsetup{format=myformat}
	\fbox{\includegraphics[scale= 0.4]{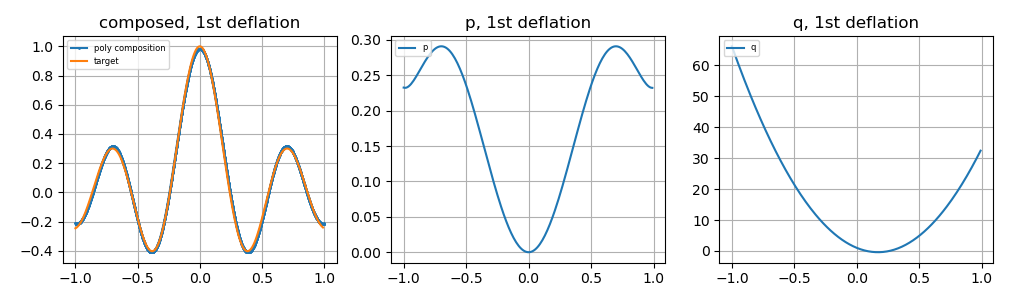}}
	\caption{Deflation 1st, $L_2$ error: 1.70e-02 }
 	\label{fig:deflation1}
\end{centering}
\end{figure}


\subsection{Comparing different deg of $p$ and $q$, for fixed degrees of freedom} \label{sec:parameter_sweep}

We approximate the shifted Bessel function \( J_1(20(x+1)) \) over the interval \([-1,1]\) to avoid the inherent symmetry. The two-layer approximation error is plotted relative to the linear least squares error, as shown in Figure \ref{fig:bescomerr}. The results indicate that the best approximation is achieved when the degree of the composite polynomial is maximized.

\begin{remark}
    This suggests that depth and width must be scaled together to control error growth. In \cite{Hanin}, Hanin demonstrates that the finite-width error of a neural network scales approximately as \(L/n\), implying that per-layer errors accumulate with the effective depth. By analyzing the cumulant recursions (Theorem~3.1 and Corollary~3.4), he shows that the network’s performance is constrained by its narrowest layer. Consequently, for a two-layer network with a fixed total number of coefficients \(T\), the optimal design is achieved by evenly distributing the coefficients between the layers, i.e., \(n_1 = n_2 = T/2\), thereby minimizing the overall error.
\end{remark}

\begin{figure}[H]
\begin{center}
	\fbox{\includegraphics[scale= 0.26]{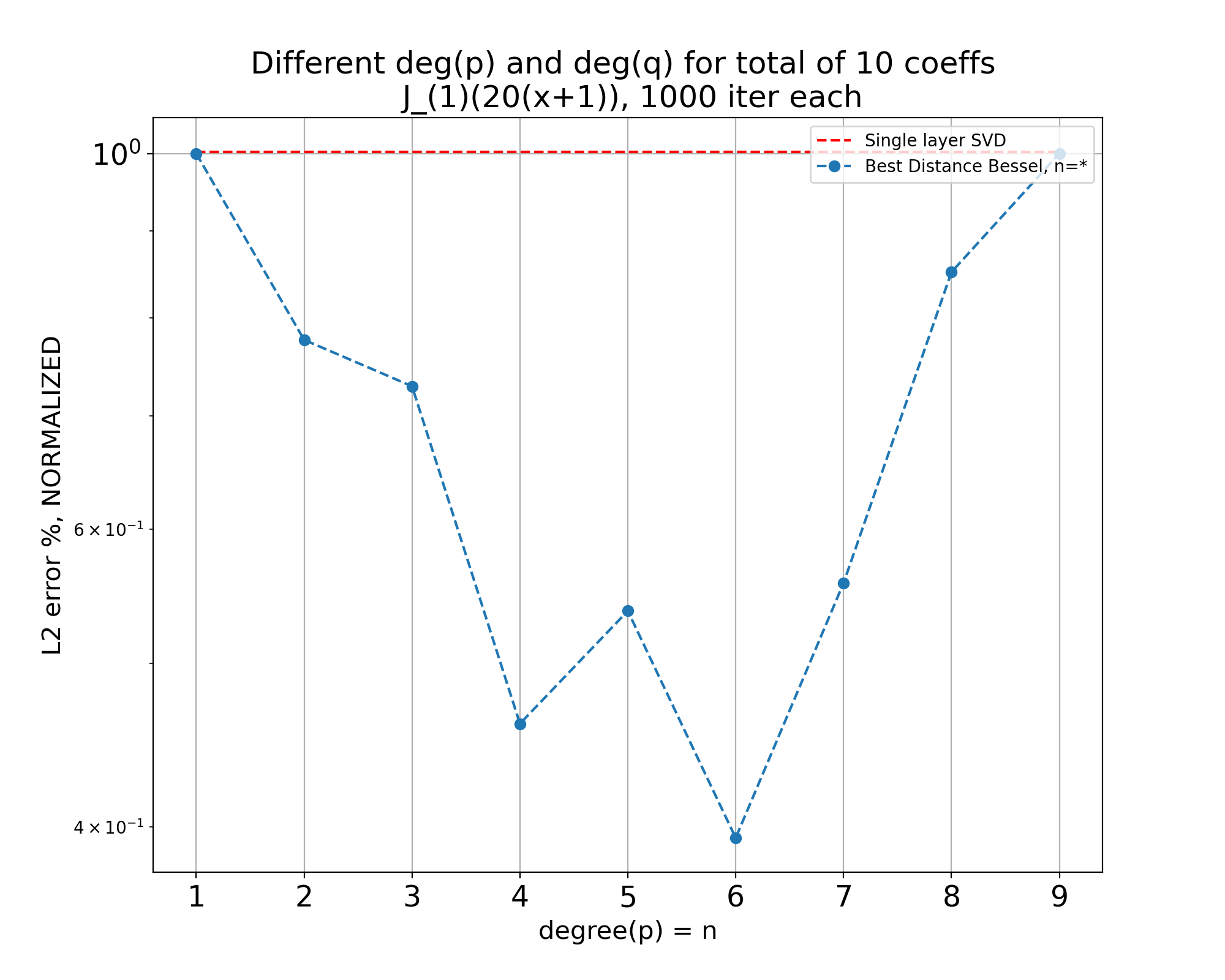}}
\end{center}
\caption{Relative error for different deg($p$), $J_{1}(20(x+1))$ BFGS $n=1000$ trials}
\label{fig:bescomerr}
\end{figure}
 
 The practical performance of such approximations relies on the effectiveness of the optimization algorithm. The left and right endpoints of the parameters sweep graph both have linear lost functions with respect to the deep polynomial coefficients (either the outer or inner polynomial is linear). In this case, we can find a one-to-one map between the deep polynomial coefficients to the linear least squares problem provided that $b_d$ the outer polynomial leading coefficients is not zero. 
 
 Delving into the specific example of Figure \ref{fig:j4030param}, we define $p_{\text{left}}(x) = 1x + 0$ (inner polynomial) after normalization and $q_{\text{left}}(x) = b_{27}x^{27} + b_{26}x^{26} + \cdots + b_1 x + b_0$ (outer polynomial). Then, the composite $q(p_{\text{left}}(x)) = q(p)_{\text{left}} = q_{\text{left}}$. 

For the right endpoint, the polynomial corresponds to $p_{\text{right}}(x) = 1x^{27} + \widetilde a_{26}x^{26} + \cdots + \widetilde a_1 x + 0$ (inner) and $q_{\text{right}}(x) = \widetilde b_1 x + \widetilde b_0$ (outer). The composite $q(p)_{\text{right}} = \widetilde b_1(1x^{27} + \widetilde a_{26}x^{26} + \cdots + \widetilde a_1 x + 0) + \widetilde b_0$. The right composite polynomial is a proper subset of the left composite polynomial since $\widetilde b_1$ cannot be $0$ to match the coefficients by the degree of the two polynomials. This difference is also evident in the gradients of the two loss functions, where $$\widetilde b_1 \cdot  \frac{\partial}{\partial b_i}\left[ \frac{1}{2}||q(p)_{\text{left}} - J_{40}(30x)|| \right ] = \frac{\partial}{\partial \widetilde a_i}\left [ \frac{1}{2}||q(p)_{\text{right}} - J_{40}(30x)|| \right ] $$ for $1 \leq i \leq 26$. Furthermore, the partial derivatives with respect to the highest degree of both composites differ, with the inner polynomial being linear, i.e., $\frac{\partial}{\partial b_{27}} q(p)_{\text{left}} = x^{27}$, and the outer polynomial being linear, i.e, $\frac{\partial}{\partial \widetilde b_{1}} q(p)_{\text{right}} = x^{27} + \widetilde a_{26}x^{26} + \cdots + \widetilde a_1 x$.

Using this map we can use the coefficients from linear least squares as initial conditions for the optimization in the case where either the outer or inner polynomial is linear (see Figure \ref{fig:j4030param}). 

\begin{figure}[H]
\begin{center}
	\fbox{\includegraphics[scale= 0.26]{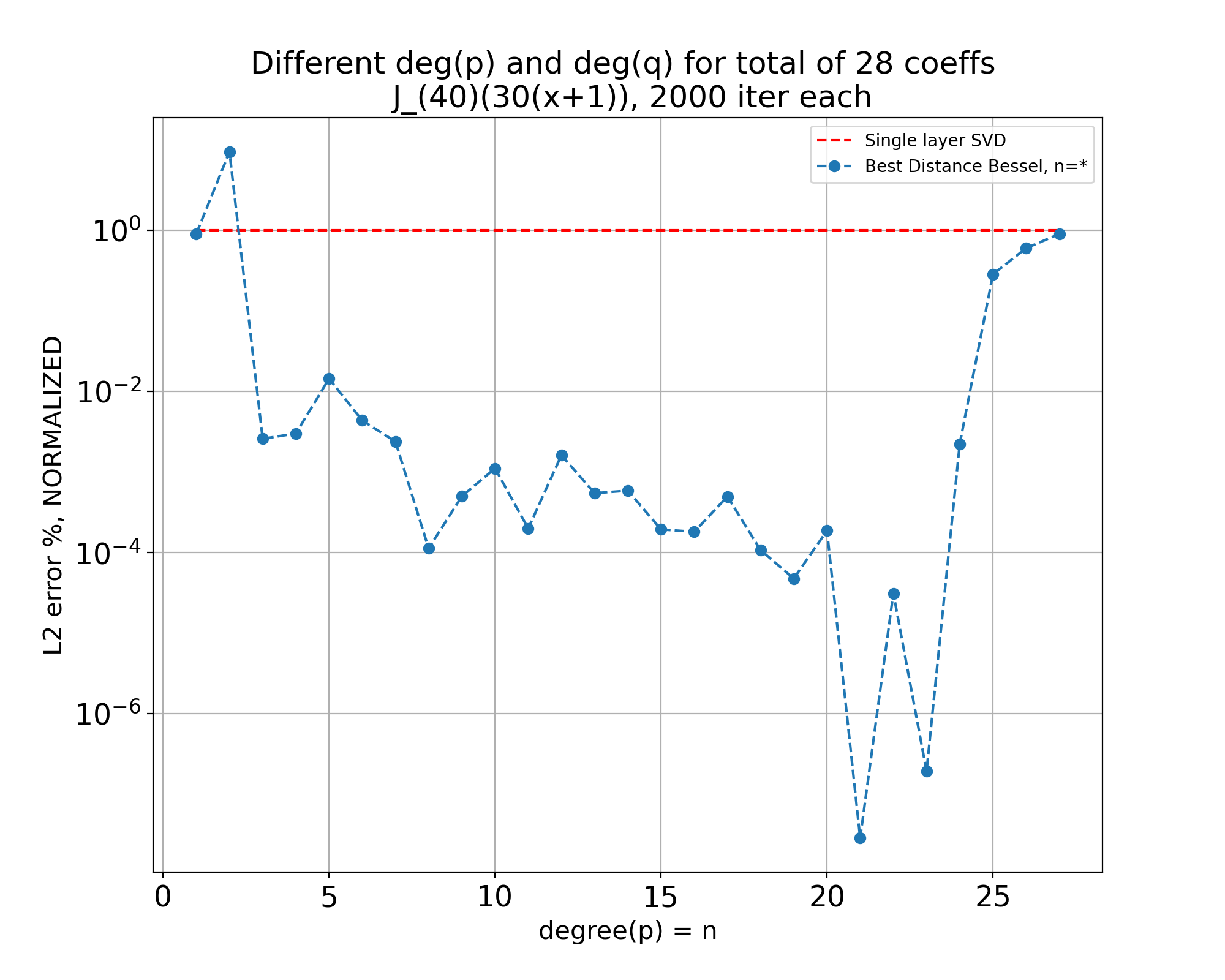}}
\end{center}
\caption{Relative error for different deg($p$), $J_{40}(30(x+1))$ BFGS $n=2000$ trials}
\label{fig:j4030param}
\end{figure}

To get further insights as to how the slight difference in the lost function of the two linear cases affects the optimization we can examine a simplified model with the same features. This simplified problem involves approximating $J_{40}(30x)$ using a 2-layer composite polynomial of degree 2 with 2 degrees of freedom. This leads to two sets of functions. In the first pair, we have $p_{\text{deg1}} = x$ and $q_{\text{deg2}} = b_2x^2 + b_1x$, resulting in a composition of $q(p(x))_{\text{deg1,deg2}} = b_2x^2 + b_1x$. In the second pair, we have $p_{\text{deg2}} = x^2 + a_1x$ and $q_{\text{deg1}} = b_1x$, which leads to a composition of $q(p(x))_{\text{deg2,deg1}} = b_1x^2 + \boldsymbol{b_1a_1}x$. For simplicity, we consider them as zero polynomials (i.e., with no constant term). The first composite polynomial $q(p(x))_{\text{deg1,deg2}}$ corresponds to the familiar linear least squares formulation. Conversely, the latter $q(p(x))_{\text{deg2,deg1}}$ has the bolded coefficients ``coupled", resulting in a very ``flat" loss function along $a_1$ when $b_1$ is close to zero (refer to Figure \ref{fig:lostplot}).

\begin{figure}[H]
\begin{center}
	\fbox{\includegraphics[scale= 0.22]{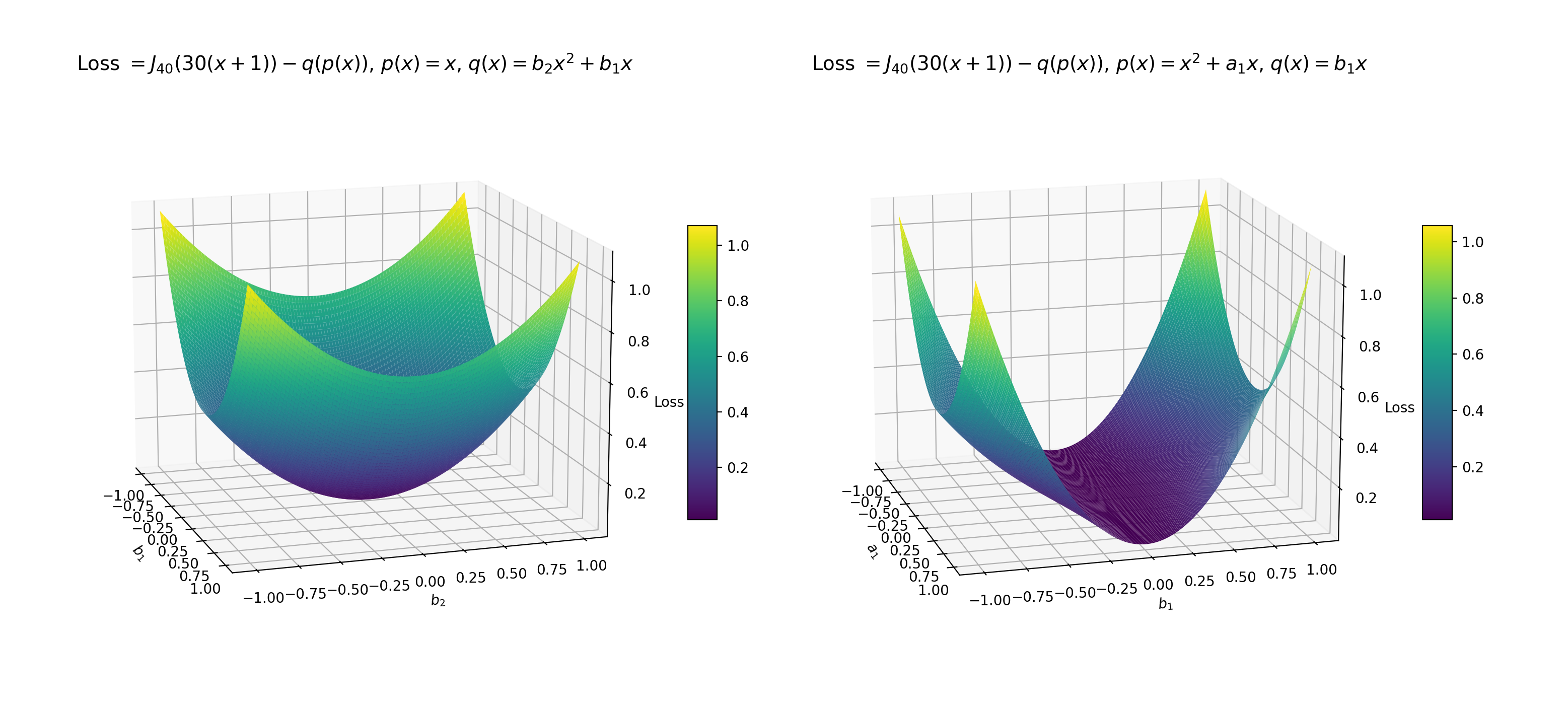}}
\end{center}
\caption{Lost function plot of $J_{40}(30(x+1))$ under some simplification}
\label{fig:lostplot}
\end{figure}


\subsection{Ensemble of random initial condition} \label{sec:ensemble}

\begin{figure}[H]
\begin{center}
	\fbox{\includegraphics[scale= 0.072]{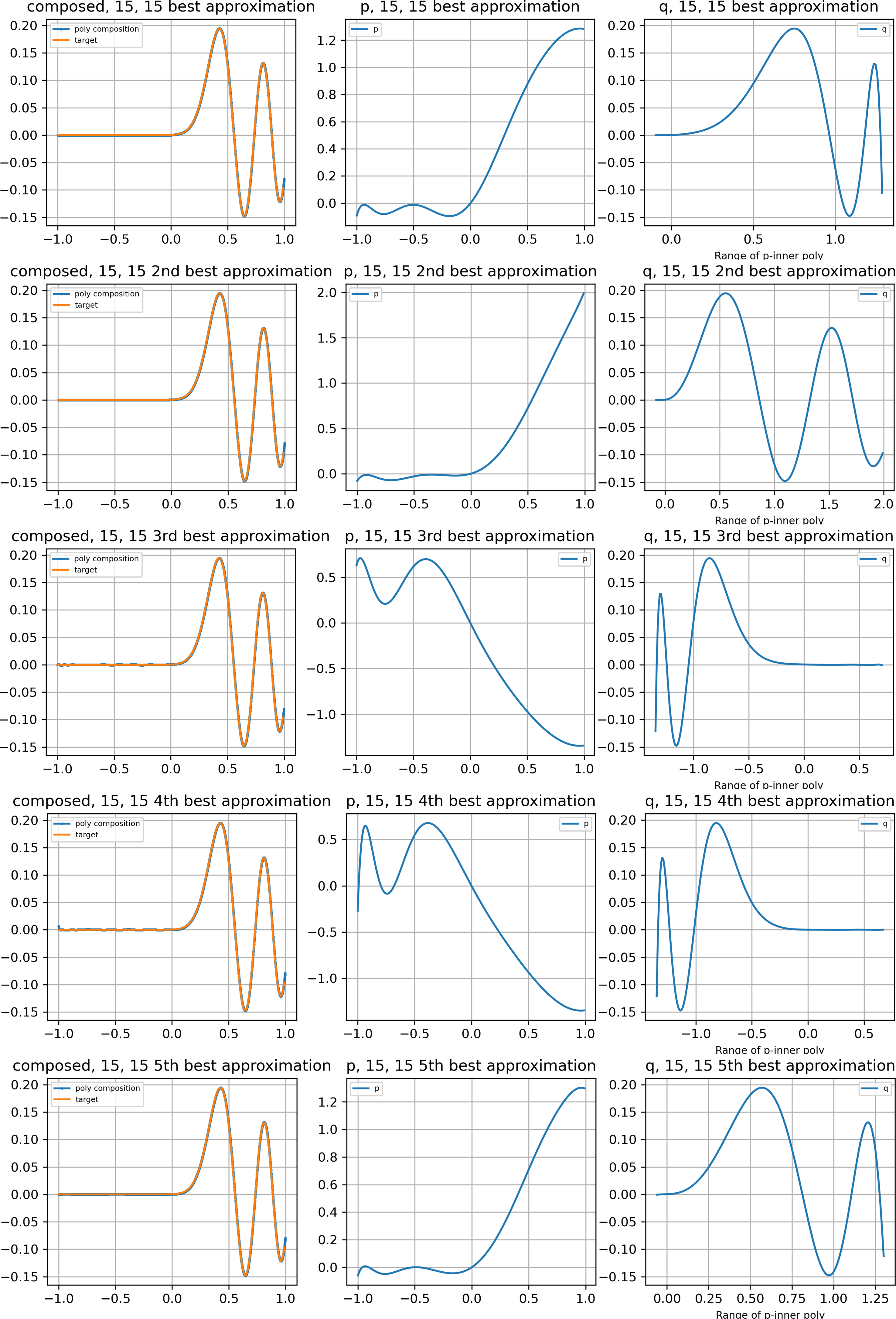}}
\end{center}
\caption{Best 5 polynomial approximation to $J_{40}{(30(x+1))}$ from $n=2000$ random trials} 
\label{fig:randomensemble}
\end{figure}

Figure \ref{fig:randomensemble} presents the $L_2$ errors of the top five polynomial approximations, ranked in ascending order as 1.209e-04 1.566e-04 3.121e-04 3.910e-04, and 3.915e-04. A closer examination reveals that some of these polynomial approximations, obtained from 2000 random trials, exhibit similarities (e.g., the 1st, 2nd, and 5th are alike, as are the 3rd and 4th). This suggests that local minima that yield good performance need not be concentrated in specific regions. Consequently, we need to explore the lost landscapes, solely employing techniques such as deflation, simulated annealing or basin-hopping without random initialization might not improve the results, which aligns with empirical observations.

\begin{figure}[H]
\begin{center}
	\fbox{\includegraphics[scale= 0.298]{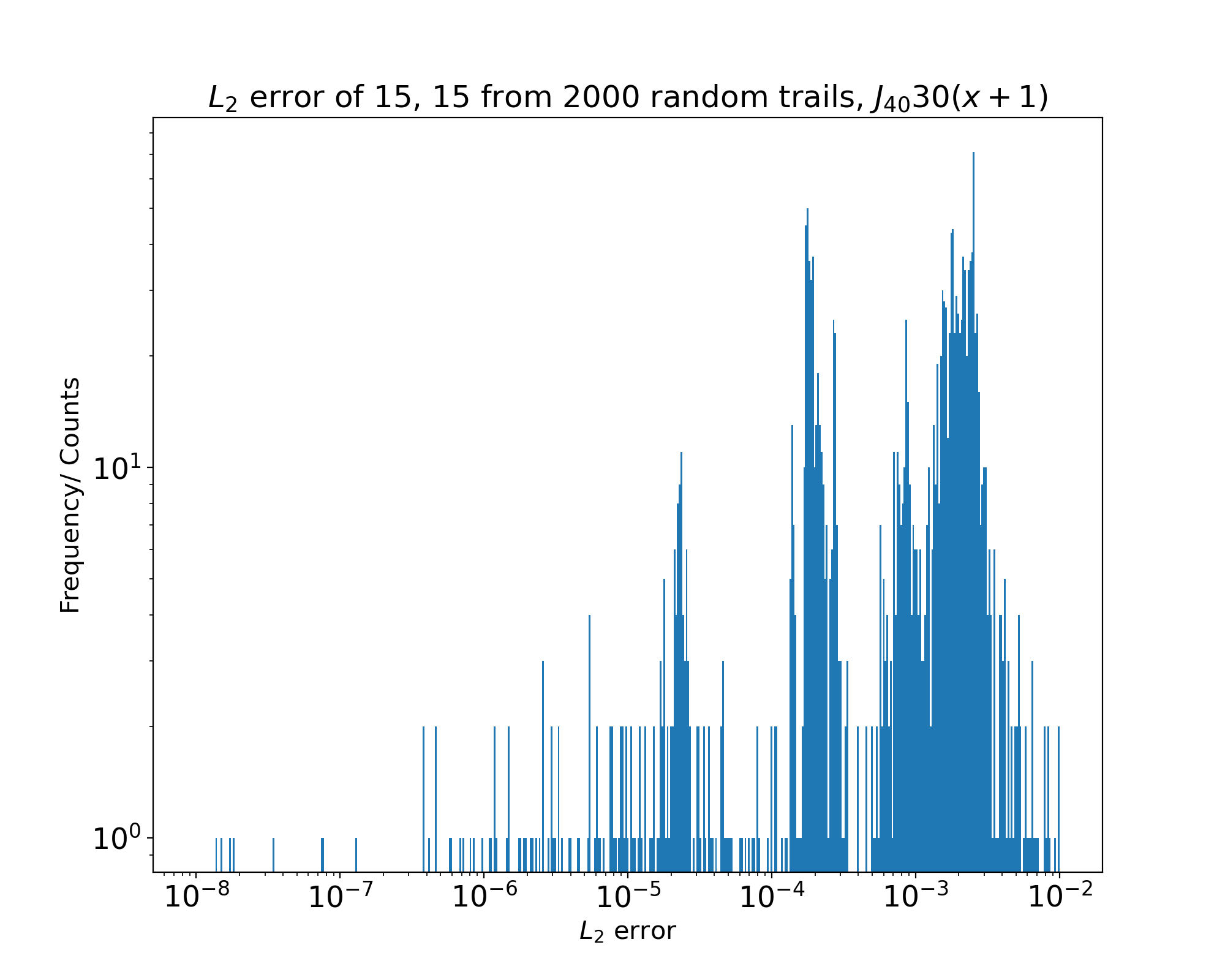}}
\end{center}
\caption{Histogram of $L_2$ errors of $J_{40}{(30(x+1))}$ from $n=2000$ random trials} 
\label{fig:hist}
\end{figure}

Figure \ref{fig:hist} strengthens our prior conjecture that finding the best optimizer is a rare occurrence. The presence of three distinct ``spikes" in the frequency counts suggests that the local minima corresponding to suboptimal approximations have comparatively larger basins of attraction.

\begin{figure}[H]
\begin{center}
	\fbox{\includegraphics[scale= 0.1]{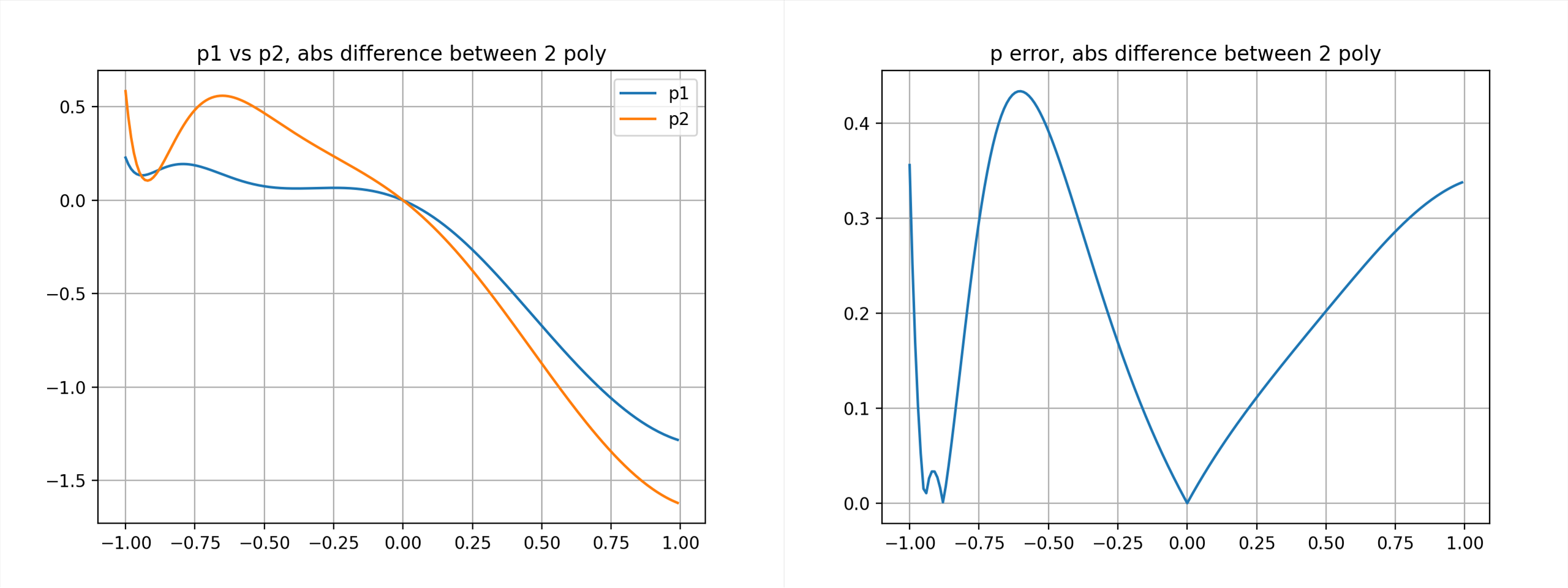}}
\end{center}
\caption{Best vs 2nd best p-(inner) to $J_{40}(30(x+1))$ from $n=10000$ random trials} 
\label{fig:optimalp}
\end{figure}

\begin{figure}[H]
\begin{center}
	\fbox{\includegraphics[scale= 0.1]{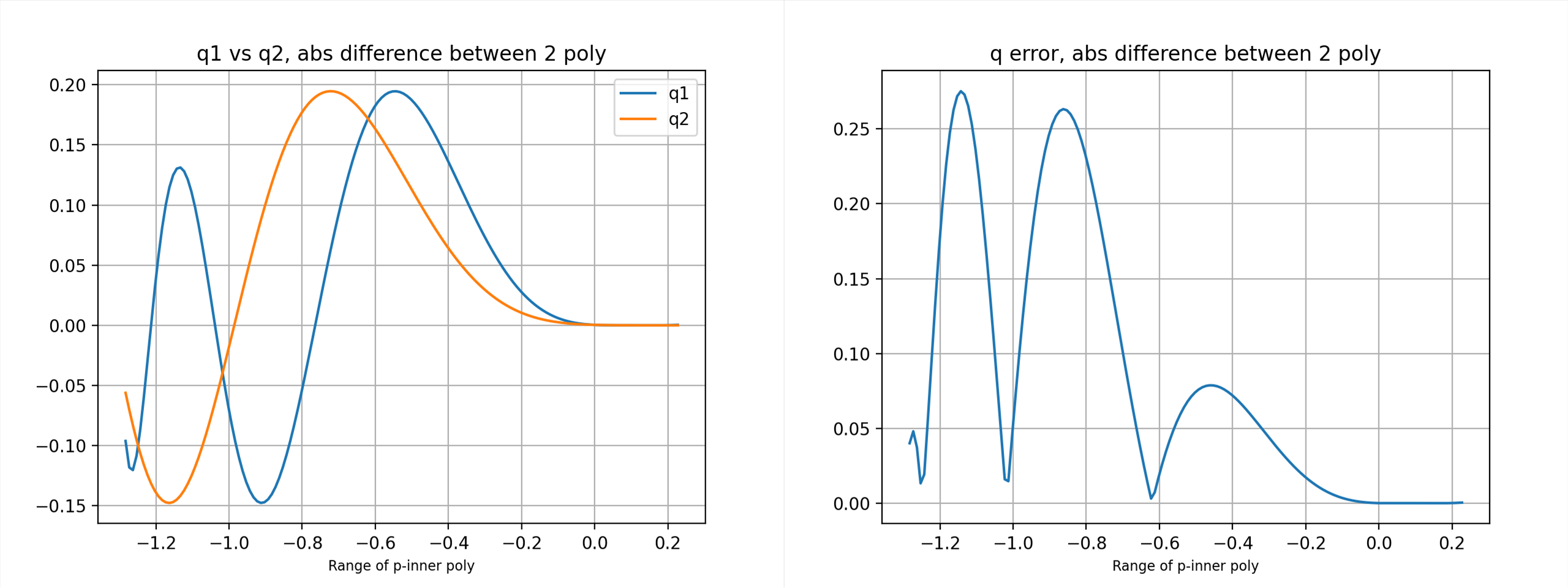}}
\end{center}
\caption{Best vs 2nd best q-(outer) to $J_{40}(30(x+1))$ from $n=10000$ random trials} 
\label{fig:optimalq}
\end{figure}

We illustrate the sensitivity of the optimization by taking the best and second best 2-layer {\it{deep}} approximation out of 10,000 random trials to $J_{40}(30(x+1))$, and observing their differences. The $L_2$ error of the best approximation is 3.424774e-05, and the $L_2$ error of the second-best approximation is 3.456158e-05. Their relative's difference is in the order of 9e-3.

When examining the individual inner and outer layers in Figure \ref{fig:optimalp}, \ref{fig:optimalq}, we note that we only use the range of $y=p(x)$ -inner as the domain of $q(y)$ -outer. Notably, the maximum absolute errors are approximately 0.4 and 0.3 respectively, each being 1000 times worse than the composite error. While the two best approximations differ pointwise, they share an overarching shape. The effects of small perturbation in the deep polynomial result in major differences in the individual polynomial layers can be seen by the chain rule. This sensitivity resembles the characteristics observed in neural networks, where the manifold illustrating the output of ReLU networks exhibits space-filling properties \cite{Devore}. These properties, enable the approximation of broader function classes using relatively few parameters but come at the expense of algorithmic stability \cite{Devore}. Thus, delving into the establishment of numerical stability during parameter selection becomes essential, specifically exploring how the optimization landscape is influenced by factors such as the target function, the number of deep layers, and the degrees of individual layers.

\section{Conformal maps as preconditioners} \label{sec:conformalmap}


\subsection{Convergence of polynomial interpolation}
The Runge phenomenon may arise when employing high-degree polynomial interpolation over equispaced nodes, resulting in oscillations or significant errors, particularly near endpoints. 

In general, it is the analyticity or lack thereof in the target function \( f \) around the domain of approximation that affects the rate convergence of an approximation. From potential theory the region of analyticity needed for polynomial interpolation is the region enclosed by the equipotential of $u(s) = \int_{-1}^1\log|x-\tau|d\mu(\tau)$, where $\mu(\tau)=1/2$ for equally spaced grids in [-1,1]. Which can be written in closed form as 

\begin{equation} \label{e:equi_pot}
    u_{equispaced}(s) = -1 + \frac{1}{2} \text{Re} \left[ (s+1) \log(s+1) - (s-1) \log(s-1) \right].
\end{equation}
We will denote the area enclosed by the equipotential curve of \eqref{e:equi_pot} that crosses [-1,1] as the Runge region. ``For the interpolants to a function \( f \) at equispaced nodes to converge as \( n \rightarrow \infty \) for all \( x \in [-1, 1] \), \( f \) must be analytic not only on \([-1, 1]\) but throughout the Runge region, which intersects the real axis at \( \pm 1 \) and the imaginary axis at \( \pm 0.52552491457 \ldots i \)"  \cite{tref}.

It is known that nodes that are clustered near the endpoints such as roots or extrema of orthogonal polynomials (i.e Chebyshev, Legendre, or Jacobi) are good for interpolation. The nodes that arise from orthogonal polynomials are asymptotically distributed to the {\it Chebyshev measure}, $\mu(\tau) = 1/(\pi \sqrt{1-\tau^2)}$. Moreover, polynomial interpolants to an analytic function on [-1, 1] converge exponentially when the nodes employed are from the Chebyshev measure \cite{tref}. The Chebyshev potential is given by

\begin{equation}
    u_{cheb}(s)= \int_{-1}^1 \frac{\log|s-x|}{\pi \sqrt{1-x^2}}dx = \log|s+i\sqrt{1-s^2}| - \log 2,
\end{equation}
which is, $-\log2$ on [-1, 1]. Thus the region of analyticity needed for convergence is just the region of interpolation, hence these nodes are effective against the Runge phenomenon. In the finite case, where we want to find $n+1$ points for polynomial interpolation, {\it Fekete, Leja} and {\it Fejer} points had been the topic of study. {\it Fekete} and its approximate {\it Leja} points are optimal in some energy sense, see \cite{tref} for definition and details. Nevertheless, it is with this perspective that we wanted to investigate whether a composition of polynomials (i.e. a polynomial map) could achieve similar performance with various numbers of interpolation points.


\subsection{An example conformal map: $(z+z^{3})/2$}

Our goal is to interpolate the Runge function,
\begin{equation}
	f(x)= \frac{1}{1+ax^2},
\end{equation}
with equispaced nodes after a conformal change of variable,
\begin{equation} \label{t1}
	x(z)= \frac{z+z^{3}}{2}
\end{equation}

The Runge function has poles at $\pm \frac{i}{\sqrt{a}}$. We want to show that $f(x(z))$ has poles in the $z$ coordinate that are outside of the Runge region. To find the pole locations, we use Cardano's formula by solving
\begin{equation} \label{e:x+x3map}
	1+a\left(\frac{z+z^3}{2}\right)^2 = 0
\end{equation}
for \(z\) after the conformal map \eqref{t1}.

We rewrite \eqref{e:x+x3map} as
\begin{equation} \label{e:x+x3fractored}
	\left[1+i\sqrt{a}\left(\frac{z+z^3}{2}\right)\right]\left[1-i\sqrt{a}\left(\frac{z+z^3}{2}\right)\right] = 0.
\end{equation}

Solving for the first bracket,
\begin{equation}
	1+i\sqrt{a}\left(\frac{z+z^3}{2}\right) = 0,
\end{equation}
we get
\begin{equation}
	z(z^2+1)= \frac{2}{\sqrt{a}} i.
\end{equation}

Assuming \(z=bi\),
\begin{equation}
	bi(1-b^2)= \frac{2}{\sqrt{a}} i \quad \Rightarrow \quad b(1-b)(1+b)=\frac{2}{\sqrt{a}},
\end{equation}
we obtain
\begin{equation}
	b^3 -b = -\frac{2}{\sqrt{a}}.
\end{equation}

Setting \(b=x\) and applying Cardano's formula, we let \(w\) be the unit root of \(x^3-1=0\) (i.e., \(w = e^{\frac{2\pi}{3}i}\)). We obtain three solutions from the first bracket of \eqref{e:x+x3fractored}:
\[
\begin{cases}
	x_k= \left( \sqrt[3]{\sqrt{\frac{1}{a}-\frac{1}{27}} -\sqrt{\frac{1}{a}}} + \frac{1}{3}\left(\sqrt{\frac{1}{a}-\frac{1}{27}} -\sqrt{\frac{1}{a}}\right)^{-\frac{1}{3}} \right)w^{k-1}, \\
	z_k= x_k \cdot i,
\end{cases}
\]

For \(a =25\), the calculation gives the pole \(|z_{k=1}|= 1.1597\ldots\), which is outside the Runge region of approximately 0.52. Therefore, equispaced interpolation in the \(z\) coordinate converges (see Figure \ref{fig:xplusx^3}). 

Figure \ref{fig:chebmap} compares the performance of this map with the Chebyshev map, \(x = \cos(z)\) on the Runge function.

\begin{figure}[H]
\begin{centering}
\captionsetup{format=myformat}
	\includegraphics[scale= 0.4]{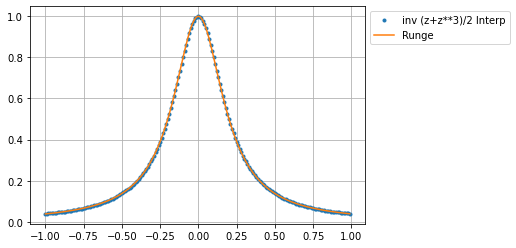}
\caption{$(z+z^3)/2$, $n=30$ interpolation points, $L_2$ error: 2.81e-04.}
	\label{fig:xplusx^3}
\end{centering}
\end{figure}

\begin{figure}[H]
\begin{centering}
\captionsetup{format=myformat}
	\includegraphics[scale= 0.286]{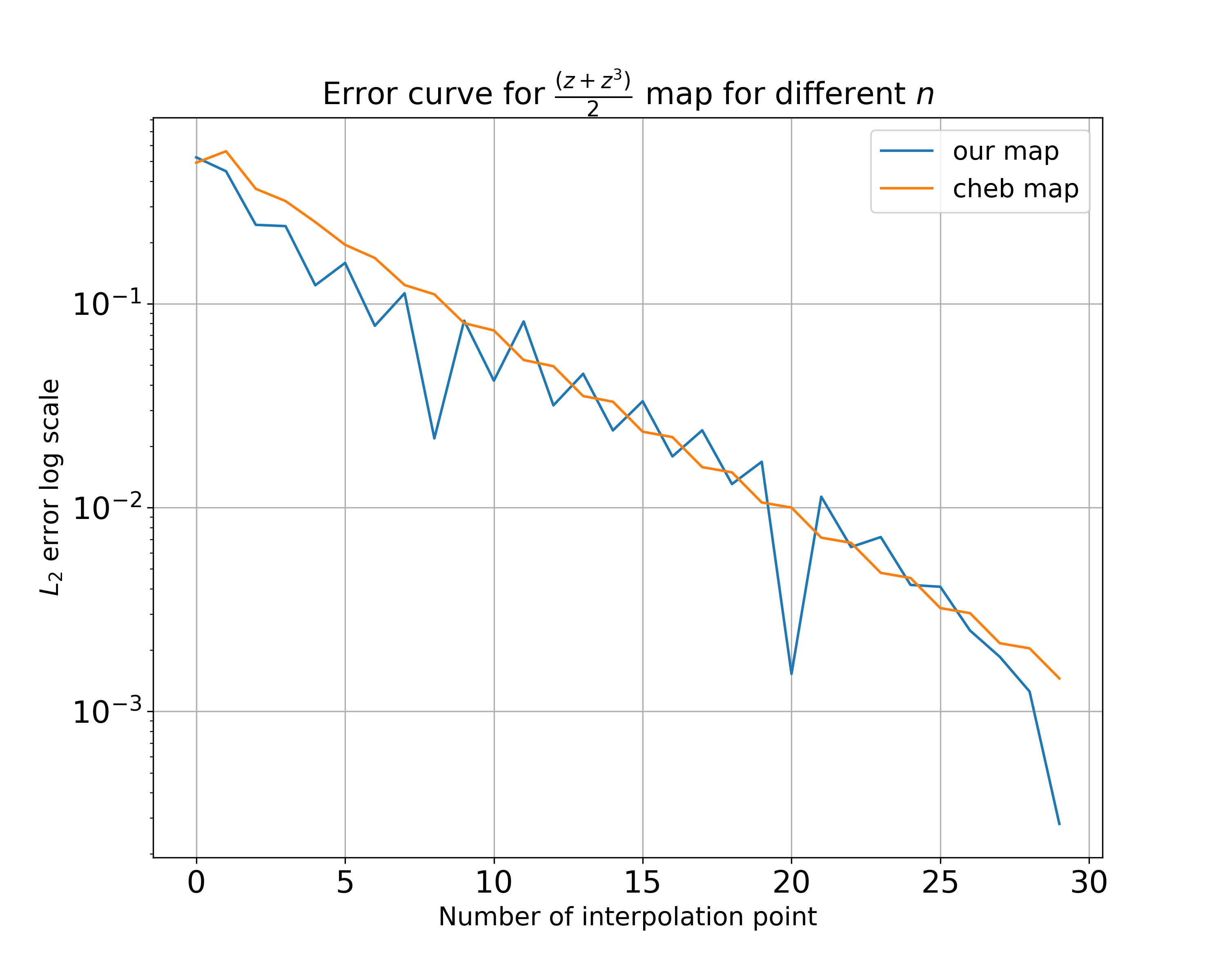}
\caption{Our map vs.\ Chebyshev for Runge function\\ $L_2$ error at n=30: 2.81e-04($\frac{z+z^3}{2}$), 1.45e-03(Cheb).}
	\label{fig:chebmap}
\end{centering}
\end{figure}


\section{Applications} \label{sec:applications}

Degree-optimal polynomial \cite{elias}, a subset of deep polynomial that optimized over the number of non-scaler multiplication is used for the evaluation of function of matrices. Degree-optimal polynomial is shown to outperform Pad\'{e} based technique and competative with current state-of-the-art methods for the square root \cite{elias} and logarithm \cite{elias24}.

\section*{Acknowledgments} 
The author expresses his sincere gratitude to Jonathan Goodman for his mentorship during the NYU SURE program and acknowledges the SURE program for partial financial support during the summer of 2022. The author is also deeply grateful to Nick Trefethen for his detailed feedback on the draft, insightful discussions on the approximation of the absolute value function, and for bringing the work of Gawlik and Nakatsukasa to his attention.


\end{document}